\documentclass{siamart250211}

\usepackage{microtype}
\usepackage{amssymb}
\usepackage{bm}

\usepackage{epstopdf}
\ifpdf
  \DeclareGraphicsExtensions{.eps,.pdf,.png,.jpg}
\else
  \DeclareGraphicsExtensions{.eps}
\fi
\usepackage{xcolor}
\usepackage{tikz}
\usepackage{float}
\usepackage{graphicx}
\usepackage[caption = false]{subfig}
\usepackage{algorithmic}
\newsiamremark{remark}{Remark}
\newsiamthm{coro}{Corollary}
\newsiamthm{prop}{Proposition}

\newcommand{\newsiamdefn}[2]{
  \theoremstyle{plain}
  \theoremheaderfont{\normalfont\sc}
  \theorembodyfont{\normalfont}
  \theoremseparator{.}
  \theoremsymbol{}
  \newtheorem{#1}{#2}
}
\newsiamdefn{defn}{Definition}
\newsiamdefn{amp}{Assumption}

\usepackage{amsfonts}
\usepackage{mathtools}
\usepackage{nicefrac}
\usepackage{mathrsfs}
\usepackage{nicematrix}
\usepackage{amsopn}

\DeclareMathOperator*{\argmin}{arg\,min}

\def\nn{\nonumber}
\def\R{\mathbb{R}}

\def\d{\,\mathrm{d}}

\def\eps{\epsilon}

\allowdisplaybreaks[4]

\headers{Gradient Flow for GPE}{T. Chu, X. Dai, J. Wu, and A. Zhou}

\title{A gradient flow model for the Gross--Pitaevskii problem: Mathematical and numerical analysis\thanks{Submitted to arXiv.
\funding{This work was supported by the National Natural Science Foundation of China under grant 92270206 and 12571446, the Strategic Priority Research Program of the Chinese Academy of Sciences under grant XDB0640000, and the National Key R \& D Program of China under grants 2019YFA0709600 and 2019YFA0709601.}
}
}

\author{
Tianyang Chu\thanks{SKLMS, Academy of Mathematics and Systems Science, Chinese Academy of Sciences, Beijing 100190, China (\email{tchu@lsec.cc.ac.cn}).}
\and Xiaoying Dai\thanks{SKLMS, Academy of Mathematics and Systems Science, Chinese Academy of Sciences, Beijing 100190, China; and School of Mathematical Sciences, University of Chinese Academy of Sciences, Beijing 100049, China (\email{daixy@lsec.cc.ac.cn}, \email{wujing21@mails.ucas.ac.cn}, \email{azhou@lsec.cc.ac.cn}).}
\and Jing Wu\footnotemark[3]
\and Aihui Zhou\footnotemark[3]
}

\begin{document}

\maketitle

% Abstract
\begin{abstract}
This paper concerns the mathematical and numerical analysis of the $L^2$ normalized gradient flow model for the Gross--Pitaevskii eigenvalue problem, which has been widely used to design the numerical schemes for the computation of the ground state of the Bose--Einstein condensate. We first provide the mathematical analysis for the model, including the well-posedness and the asymptotic behavior of the solution. Then we propose a normalized implicit-explicit fully discrete numerical scheme for the gradient flow model, and give some numerical analysis for the scheme, including the well-posedness and optimal convergence of the approximation. Some numerical experiments are provided to validate the theory.
\end{abstract}

% Keywords
\begin{keywords}
Bose--Einstein condensate, Gross--Pitaevskii problem, gradient flow, asymptotic behavior, convergence
\end{keywords}

% MSC codes
\begin{MSCcodes}
65N12, 65N25, 65N30
\end{MSCcodes}

\section{Introduction}

Bose--Einstein condensate (BEC) is an exotic state of matter that occurs in a dilute gas of bosons cooled to temperature extremely close to absolute zero. This phenomenon was first theoretically predicted by Einstein in the 1920s \cite{Einstein1924}, building upon the foundational work of Bose \cite{Bose1924}. Experimental realization of BEC in ultracold bosonic gases was achieved in 1995, as documented in pioneering studies by Anderson et al. \cite{anderson1995observation} and Davis et al. \cite{davis1995bose}. The theoretical framework for understanding BEC is largely based on the Gross--Pitaevskii (GP) theory, independently developed by Gross \cite{gross1961structure} and Pitaevskii \cite{pitaevskii1961vortex} in the 1960s. The GP theory has been widely validated in predicting key properties of BECs and remains a cornerstone of the mathematical model for describing the BEC. In the GP theory, the quantum state of the BEC is given by the solution of the following nonlinear Schr\"odinger equation (also called the time-dependent GP equation): 
\begin{align}\label{equ: NLS in intro}
    i\psi_t = - \Delta \psi + V \psi + \beta |\psi|^2 \psi, \qquad  \text{in} \,\,\, \R^3, \quad t > 0,
\end{align}
with a given initial value $\psi(x,0) = \psi_0(x)$ satisfying $\|\psi_0\|_{L^2(\R^3)} = 1$. The wave solutions to \eqref{equ: NLS in intro} are of the form $\psi(x,t) = \varphi(x)e^{-i\mu t}$. In this sense, $\varphi(x)$ satisfies the following GP eigenvalue problem (also called the time-independent GP equation): find $(\varphi,\mu) \in H^1(\R^3) \times \R$ with $\|\varphi\|_{L^2(\R^3)} = 1$ such that in the weak sense
\begin{equation}\label{equ: GPEVP in intro}
	- \Delta \varphi + V\varphi + \beta |\varphi|^2\varphi = \mu \varphi.
\end{equation}
The ground state of BEC which is a stationary point of \eqref{equ: NLS in intro}, is described as the normalized eigenfunction $\phi$ corresponding to the smallest eigenvalue of \eqref{equ: GPEVP in intro}. Meanwhile, note that \eqref{equ: GPEVP in intro} represents the Euler-Lagrange equation of the GP energy functional which is given by 
\begin{equation}
	E(v):= \int_{\R^3} \left( \frac12 |\nabla v|^2 + \frac12 V |v|^2 + \frac{\beta}{4} |v|^4 \right) \ \text{d}x.
\end{equation}
Therefore, the ground state of a BEC can also be described as the minimizer of the following problem:
\begin{align}\label{equ: minimize problem}
    \varphi \in \argmin_{v \in S} E(v), \qquad S:=\{ v \in H^1(\R^3): \, \|v\|_{L^2(\R^3)} = 1 \}.
\end{align}
For more details about the background of the BEC and the mathematical theory of the GP theory, we refer readers to the review paper \cite{bao2013mathematical}.

Nowadays, the computation of the ground state of BEC is one of the fundamental objects in numerical studies of BEC and there are mainly two approaches to obtain the ground state: solving the nonlinear eigenvalue problem \eqref{equ: GPEVP in intro} and solving the minimization problem \eqref{equ: minimize problem}. The first approach, i.e., directly solving the nonlinear eigenvalue problem \eqref{equ: GPEVP in intro}, usually contains two steps. The first step is the spatial discretization such as the finite difference method, the finite element method, the spectral method, etc. There are some existing works on the error estimates for the spatial discretization, including the priori error analysis \cite{cances2010numerical,chen20111846,henning2025discrete,zhou2003analysis} and the a posteriori analysis \cite{chen20111846,dusson2017posteriori}. The second step is to solve the nonlinear algebraic problem of the form 
\begin{align}\label{equ: nonlinear algebraic problem}
    A(\mathbf{v})\mathbf{v} = \lambda \mathbf{v} \qquad \mathbf{v} \in \R^{N}.
\end{align}
The most common iterative technique for solving \eqref{equ: nonlinear algebraic problem} is the self-consistent field iteration (cf. \cite{ cances2000SCF, cances2021convergence, cances2000on,dion2007ground, upadhyaya2021density}), where a linearized eigenvalue problem is solved in each iteration. A systematic and comprehensive introduction of this approach can be found in a recent review paper \cite[Sections 3 and 4]{henning2025gross}. For the second approach, i.e., solving the problem \eqref{equ: minimize problem}, there are many existing works such as directly minimizing the energy functional in the finite element space \cite{bao2003ground}, the Riemannian optimization method \cite{altmann2022energy, altmann2024riemanniannewton, danaila2017computation}, the preconditioned nonlinear conjugate gradient method \cite{antoine2017efficient}, the regularized Newton method \cite{wu2017regularized}, the $J$-method \cite{altmann2021j, jarlebring2014inverse}, the gradient-flow based methods,  etc. 

In this paper, we study the so-called gradient flow based method, which is widely used to solve \eqref{equ: GPEVP in intro}. The main idea of the gradient flow approach is to introduce an artificial notion of time $t$ and a corresponding time-dependent state $\phi(t)$, then construct a continuous (projected) gradient flow model (an evolution equation with respect to $\phi(t)$), for which $\lim_{t\to\infty}\phi(t)$ is the solution of \eqref{equ: GPEVP in intro}. After that, we can get several methods by discretizing the gradient flow model in time. Roughly speaking, there are mainly two types of the gradient flow model: the $L^2$ projected gradient flow and the $H^1$ projected gradient flow. The $L^2$ projected gradient flow was first introduced in \cite{bao2004computing}, called continuous normalized gradient flow (CNGF), which was used as a mathematical justification for 
% gradient flow with discrete normalization (GFDN) corresponding to 
the imaginary time evolution method \cite{aftalion2001vortices,cerimele2000numerical,chiofalo2000ground}. Due to the simplicity and efficiency of the gradient flow with discrete normalization (GFDN) and CNGF models \cite{bao2004computing}, based on which lots of numerical methods have been applied to compute the ground state \cite{bao2013mathematical,bao2006efficient,bao2004computing,bao2007mass,liu2021normalized,wang2014projection,zhuang2019efficient}. Convergence of GFDN iteration (cf. \cite[Definition 4.1]{henning2023dependency}) in 1D for focusing nonlinearities ($\beta < 0$) was considered in \cite{faou2018convergence}, and was considered in \cite{henning2023dependency} for higher dimensions with defocusing nonlinearities ($\beta > 0$). Several $H^1$ gradient flows for the GP eigenvalue problem have been proposed in \cite{danaila2010new,henning2020sobolev,  kazemi2010minimizing}. With these $H^1$ Sobolev gradient flow models, combining a forward Euler discretization, several semi-discrete methods have been proposed in \cite{danaila2010new,henning2020sobolev,  kazemi2010minimizing} with numerical analysis in \cite{chen2024convergence,henning2023dependency,henning2020sobolev,zhang2022exponential}. Furthermore, combining with some spatial discretization, fully discrete methods have been proposed in \cite{chen2024fullydiscretizedsobolevgradient,heid2021gradient}. 

The $L^2$-normalized gradient flow, as a continuous model provided to design the numerical schemes, its mathematical properties such as well-posedness and asymptotic behaviors are of importance to guarantee the reasonability of the corresponding numerical schemes. Compared to the $H^1$-normalized gradient flow, there are few works on the mathematical analysis for the $L^2$ gradient flow. With the absence of the potential function, i.e., $V \equiv 0$, for the case of 1D with $\beta<0$, Faou and J\'{e}z\'{e}quel \cite{faou2018convergence} proved the exponential convergence in $L^2$ space under the assumption that the initial value is sufficiently close to the ground state. Antonelli et al. \cite{antonelli2024existence} proved the well-posedness of the $L^2$ gradient flow with general nonlinearity, i.e., $|\varphi|^{2\sigma}\varphi$, in higher dimensions. Different from the $H^1$ case, the mathematical analysis of the $L^2$ gradient flow, especially the well-posedness of the model, requires the theory of PDEs, while the $H^1$ case only depends on the ``ODEs" knowledge.

One of the objectives of this paper is to give a mathematical analysis for the $L^2$ normalized gradient flow of the GP eigenvalue problem. Following the recent work \cite{antonelli2024existence}, we obtain the well-posedness of the $L^2$ normalized gradient flow with $L^{\infty}$ potential. In addition to the well-posedness, similar to the $H^1$ case \cite{henning2020sobolev, kazemi2010minimizing}, under some assumptions on the regularity of the solution, we also get several exponential convergences of the solution to the $L^2$ normalized gradient flow as the time approaches infinity. Another objective of this paper is to consider the numerics of the $L^2$ normalized gradient flow, including the numerical scheme and its convergence analysis. We propose a normalized implicit-explicit fully discrete scheme for the gradient flow model. A similar temporal semi-discrete scheme has been introduced in \cite{liu2021normalized}, combining the spatial Fourier-pseudospectral discretization. Numerical experiments in \cite{liu2021normalized} show that this scheme is very efficient since it only needs to solve a linear elliptic equation with constant coefficients in each iteration. For the convergence analysis, different from \cite{faou2018convergence, henning2023dependency}, which considered the convergence and convergence rate of the iteration solution to the ground state with respect to the iteration numbers, our analysis concentrates on the approximation of the gradient flow model with respect to the mesh size (cf. Theorem \ref{thm: L2 convergence}), and our numerical analysis can be naturally extended to similar methods such as numerical schemes in \cite{liu2021normalized}. One of the difficulties in the convergence analysis is to estimate the error of the normalized step, which is overcome in a geometric view (cf. Lemma \ref{lem: key lemma in geometry}). We obtain also an optimal convergence in $L^{\infty}(0,T;L^2(\Omega))$ and a suboptimal convergence in $L^{\infty}(0,T;H^1(\Omega))$.

The organization of this paper is as follows. In Section \ref{sec: Preliminaries}, we introduce the basic knowledge of the GP eigenvalue problem and its gradient flow. In Section \ref{sec: Mathematical analysis}, we give some mathematical analysis for the $L^2$ gradient flow of the GP eigenvalue problem, including the well-posedness, the exponential convergence of the energy and the exponential convergence of the solution to the ground state in $L^2$ norm, as the time approaches infinity. In Section \ref{sec: Numerical analysis}, we propose a numerical scheme for the $L^2$ gradient flow of the GP eigenvalue problem and carry out the convergence analysis with respect to the mesh size. In Section \ref{sec: numerical experiments}, we provide two numerical experiments to validate our theoretic results obtained in Section \ref{sec: Numerical analysis}, and give a conclusion in Section \ref{sec: conclusion}.

Throughout this paper, the letter $C$ denotes a generic constant that could be different in different occurrences.

\section{The gradient flow model}\label{sec: Preliminaries}

\subsection{The Gross--Pitaevskii problem}\label{sec GP problem}

The ground state of BEC is modeled by the following time-independent GP equation, also called the GP eigenvalue problem: Find $(\varphi, \lambda) \in H^1(\R^3) \times \R$ with  $\|\varphi\|_{L^2(\R^3)} = 1$ such that in the weak sense
\begin{align}\label{equ: TIGPE}
	- \Delta \varphi + V\varphi + \beta |\varphi|^2\varphi = \lambda \varphi,
\end{align}
where $V$ is a real-valued time independent potential function, $\beta$ is a real number whose symbol represents the repulsive (positive)/attractive (negative) interaction. In this paper, we consider the repulsive interaction case, i.e., $\beta \ge 0$, and without loss of generality, we assume $V \ge 0$. The classical Ljusternik--Schnirelman theory \cite{zeidler1985nonlinear} shows that the problem \eqref{equ: TIGPE} has infinitely many eigenvalues $0 < \lambda_1 < \lambda_2 \le \lambda_3 \le \cdots < \infty$ and the ground state $\phi_{\text{GS}}$ is the normalized eigenfunction corresponding to the smallest eigenvalue $\lambda_{\text{GS}}:= \lambda_1$. Moreover, the ground state $\phi_{\text{GS}}$ is positive and unique up to a sign \cite{bao2013mathematical,cances2010numerical}. In addition, if $V$ satisfies $\lim_{|x| \to \infty} V(x) = + \infty$, then $\phi_{\text{GS}}$ decays exponentially fast as $|x| \to \infty$ \cite{bao2013mathematical}.

Note that \eqref{equ: TIGPE} is the Euler-Lagrange equation of the GP energy functional, which is defined as
\begin{align}
    E(v):= \int_{\R^3} \left ( \frac12 |\nabla v|^2 + \frac12 V |v|^2 + \frac{\beta}{4} |v|^4  \right ) \ \text{d}x \qquad v \in H^1(\R^3).
\end{align}
Then the ground state of a BEC can be rephrased in terms of the solution of the following minimization problem: Find $\varphi \in S$,
\begin{align}\label{equ: minimization problem}
    E(\varphi) = \min_{v \in S} E(v).
\end{align}
The mathematical analysis of the problem \eqref{equ: minimization problem}, such as the existence, uniqueness and other properties, can be found in \cite{bao2013mathematical,cances2010numerical}.

\subsection{The \texorpdfstring{{\boldmath $L^2$}}{L2} gradient flow model}

The exponential decay property of the ground state solution mentioned in Section \ref{sec GP problem} suggests that we can consider the GP problem \eqref{equ: TIGPE} on a bounded domain $\Omega \subset \R^3$, together with a homogeneous Dirichlet boundary condition. We adopt the gradient flow approach to calculate the ground state of GP eigenvalue problem, where the ground state solution of GP eigenvalue problem as well as the global minimizer of the problem \eqref{equ: minimization problem} can be viewed as the stationary state of the gradient flow problem. The gradient flow model of the GP problem was first considered in \cite{bao2004computing} as a mathematical justification  for the classical imaginary time method, which has been used in physics literature to calculate the ground state solution of BEC \cite{aftalion2001vortices,cerimele2000numerical,chiofalo2000ground}. We recall that the Fr\'echet derivative of the GP energy functional  is given by 
\begin{equation}
	\langle E'(v), w \rangle = \int_{\Omega} \Big(\nabla v \cdot \nabla w + Vvw + \beta |v|^2 vw\Big) \ \text{d}x \quad \forall\, v,w\in H^1_0(\Omega),
\end{equation}
where the notation $\langle \cdot, \cdot \rangle$ denotes the dual pair between $H^{-1}(\Omega)$ and $H_0^1(\Omega)$.

The imaginary time approach to solve \eqref{equ: minimization problem} can be read as: for a time sequence $0 = t_0 < t_1 < t_2 < \cdots < t_n < \cdots$, 
\begin{align}
    &\phi_t  = - \frac{\delta E(\phi)}{\delta \phi} = \Delta \phi - V\phi - \beta |\phi|^2 \phi, \qquad \text{in}\,\,\, \Omega, \quad t_n < t < t_{n+1}, \quad n \ge 0, \label{equ: ITM1} \\
   & \phi(t_{n+1})  = \frac{\phi(t_{n+1}^{-})}{\|\phi(t_{n+1}^{-})\|_{L^2(\Omega)}}. \label{equ: ITM2}
\end{align}
In fact, \eqref{equ: ITM1} can be viewed as applying the steepest descent method to the energy functional $E(\phi)$ without constraint and then projecting the solution onto the unit sphere. It is mentioned in \cite{bao2004computing} that \eqref{equ: ITM1}--\eqref{equ: ITM2} can be viewed as a first-order splitting method for the following continuous normalized gradient flow (CNGF): 
\begin{align}
    &\phi_t = \Delta \phi - V \phi - \beta |\phi|^2 \phi + \mu[\phi]\phi,  && \text{in}\,\,\, \Omega \times (0,\infty), \label{equ: CNGF1} \\
    &\phi = 0,  && \text{on}\,\,\,\partial\Omega\times (0,\infty),  \\
    &\phi = \phi_0,  && t = 0 \quad \text{with} \quad \|\phi_0\|_{L^2(\Omega)} = 1, \label{equ: CNGF2}
\end{align}
where 
\begin{align}\label{equ: CNGF3}
    \mu[\phi]:= \frac{1}{\|\phi\|_{L^2(\Omega)}^2} \int_{\Omega}\left( |\nabla \phi|^2 + V |\phi|^2 + \beta |\phi|^4\right) \ \text{d}x.
\end{align}
\eqref{equ: CNGF1}--\eqref{equ: CNGF3} is called the $L^2$ gradient flow model of the GP problem,
and its mathematical analysis will be presented in the next section.

\begin{remark}[Projected gradient flow approach]
    The gradient flow model \eqref{equ: CNGF1}--\eqref{equ: CNGF3} can also be derived from the projected gradient flow approach. In the context of projected Sobolev gradient flow, the critical points of the energy $E(\cdot)$ can be identified by constructing appropriate gradient flows of the form 
\begin{equation}\label{equ: 111111111}
	z'(t) = - P_{z,X}(\nabla_X E(z(t))),
\end{equation}
where $\nabla_X E$ is the Sobolev gradient of the energy functional and $P_{z,X}$ is the projection operator. With the choice of $X = L^2(\Omega)$, \eqref{equ: 111111111} coupled with some initial value $\phi(0) \in H_0^1(\Omega) \cap H^2(\Omega)$ is exactly \eqref{equ: CNGF1}--\eqref{equ: CNGF3} \cite[Section 2.2.1]{henning2020sobolev}.
\end{remark}

\section{Mathematical analysis}\label{sec: Mathematical analysis}

In this section, we consider the following nonlinear, parabolic PDEs:
\begin{equation}\label{equ: PDEofL2GF}
\begin{cases}
	\phi_t = \Delta \phi - V\phi - \beta |\phi|^2 \phi + \mu[\phi] \phi, \qquad & \text{in}\,\,\, \Omega \times (0,\infty),   \\
	 \phi = 0, \qquad & \text{on}\,\,\,\partial\Omega\times (0,\infty), \\
		\phi = \phi_0, \qquad & t = 0, \qquad \|\phi_0\|_{L^2(\Omega)} = 1,
		\end{cases}
\end{equation}
where $V \ge 0$ and $\beta \ge 0$ and $\mu[\phi]$ is defined in \eqref{equ: CNGF3}. For simplicity, we assume $\Omega \subset \R^3$ to be a bounded convex domain with smooth boundary. First, we recall the properties of \eqref{equ: PDEofL2GF} (see Theorem 2.5 of \cite{bao2004computing}).
\begin{prop}
    The solution of \eqref{equ: PDEofL2GF} satisfies normalization conservation and energy diminishment, i.e., 
    \begin{align*}
        & \|\phi(\cdot,t)\|_{L^2(\Omega)} = \|\phi_0\|_{L^2(\Omega)} = 1, \qquad t \ge 0, \\
        & \frac{\mathrm{d}}{\mathrm{d}t}E(\phi) = -\|\phi_t(\cdot,t)\|_{L^2(\Omega)}^2 \le 0, \qquad t \ge 0.
    \end{align*}
\end{prop}

Let $\{e^{t\Delta}\}_{t > 0}$ denote the semigroup generated by the Laplacian operator on $L^2(\Omega)$. The function $u = e^{t\Delta}u_0$ is the solution of the following equation
\begin{align*}
    \begin{cases}
        u_t = \Delta u, \qquad &\text{in $\Omega \times (0,\infty)$}, \\
        u = 0, \qquad &\text{on $\partial\Omega \times (0,\infty)$}, \\
        u = u_0 \in H_0^1(\Omega), \qquad &t = 0.
    \end{cases}
\end{align*}
We recall the properties of $\{e^{t\Delta}\}_{t > 0}$ \cite{antonelli2024existence,BAILLEUL20163344,Jost2013} that will be used in the latter proof.
\begin{prop}\label{prop: 1}
    For any time independent $f \in L^2(\Omega)$, we have 
    \begin{align*}
        \sup_{t>0} \|e^{t\Delta}f\|_{L^2(\Omega)} \le \|f\|_{L^2(\Omega)}.
    \end{align*}
    And for any $u_0 \in H_0^1(\Omega)$, we have 
    \begin{align*}
        \|e^{t\Delta}u_0\|_{L^2(\Omega)}^2 = \|u_0\|_{L^2(\Omega)}^2 - 2\int_0^t \|\nabla e^{s\Delta}u_0\|_{L^2(\Omega)}^2 \ \mathrm{d}s.
    \end{align*}
\end{prop}
\begin{prop}\label{prop: 2}
    For any $f \in L^1((0,\infty);L^2(\Omega))$, we have 
    \begin{align*}
        \left\| \int_0^t e^{(t-s) \Delta} f(s)\ \mathrm{d}s \right\|_{L^{\infty}((0,\infty); L^2(\Omega))} \le C \|f\|_{L^1((0,\infty);L^2(\Omega))}.
    \end{align*}
    And for any $f \in L^2((0,\infty);L^2(\Omega))$, we have 
    \begin{align*}
        \left\|\nabla \int_0^t e^{(t-s) \Delta} f(s)\ \mathrm{d}s \right\|_{L^{\infty}((0,\infty); L^2(\Omega))} \le C \|f\|_{L^2((0,\infty);L^2(\Omega))}.
    \end{align*}
\end{prop}

We then carry out the mathematical analysis of \eqref{equ: PDEofL2GF}, including the well-posedness and the asymptotic behaviors.

\subsection{Well-posedness and global convergence}

Following \cite{antonelli2024existence}, we have
\begin{theorem}\label{thm: wellposedness of model problem}
	Let $V \in L^{\infty}(\Omega)$ with $V \ge 0$ and $\beta \in \R$ be nonnegative. If the initial value $\phi_0 \in H_0^1(\Omega)\cap H^2(\Omega)$ with $\|\phi_0\|_{L^2(\Omega)}=1$, then there exists a unique global solution $\phi$ satisfying \eqref{equ: PDEofL2GF} with regularity
	\begin{equation*}
		\phi \in C([0,\infty); H^2(\Omega) \cap H_0^1(\Omega)), \qquad \phi_t \in C([0,\infty); L^2(\Omega) ).
	\end{equation*}
	Moreover, if $\phi_0 \ge 0$, then $\phi(t)$ converges to $\phi_{\text{GS}}$ strongly in $H_0^1(\Omega)$ as $t \to \infty$.
\end{theorem}
\begin{proof}
The proof will be divided into two parts.

\underline{\textit{Part I (Global well-posedness)}} We first prove that there exist $T^*>0$ and a unique local solution $\phi\in C([0,T^*],H_0^1(\Omega))$. Let $M>0$, $N > 0$, $T > 0$ be some constants that are to be determined. Let 
\begin{align*}
    \Lambda = \Big\{ u\in C\big([0,T];H_0^1({\Omega})\big):\, u(0) = \phi_0, \, \Vert u \Vert_{L^\infty\big([0,T];H^1({\Omega})\big)}\leq M,\,\inf_{t\in[0,T]}\Vert u \Vert_{L^2(\Omega)}\geq\frac{N}{2}\Big\}
\end{align*}
be the set equipped with the distance
\begin{align*}
    \text{d}(u,v)=\Vert u-v \Vert_{L^\infty\big([0,T];H^1({\Omega})\big)} \qquad \forall\, u,v \in \Lambda.
\end{align*}
Since $V\in L^\infty(\Omega)$, by continuous embedding $H^1(\Omega)\hookrightarrow L^4(\Omega)$ and H\"older's inequality, we conclude that there exists a constant $C_1(M,N)>0$ such that 
\begin{equation}\label{mues}
    \Vert\mu[u]\Vert_{L^\infty[0,T]}\le \frac{C}{N^2}\left(\Vert u\Vert_{L^\infty([0,T];H^1(\Omega))}^2+\Vert u\Vert_{L^\infty([0,T];H^1(\Omega))}^{4}\right)\leq C_1,\hspace{1em}u\in \Lambda.
\end{equation}
Note that for all $u,v \in \Lambda$
\begin{align*}
    \left| \mu[u] - \mu[v] \right| & \le \frac{C}{N^2} \int_{\Omega} \Big(  |\nabla u + \nabla  v||\nabla u - \nabla  v| + V(x)|u+v||u-v| + \beta |u^4 - v^4| \Big) \ \text{d}x   \\
    & \quad + \left| \|u\|_{L^2(\Omega)}^{-2} -  \|v\|_{L^2(\Omega)}^{-2}\right||\mu[u]| \\
    & \le C_2\|u-v\|_{H^1(\Omega)},
\end{align*}
here $C_2 = C_2(M,N)$, which leads to 
\begin{equation}\label{61}
    \Vert \mu[u]-\mu[v]\Vert_{L^\infty[0,T]}\leq C_2 \text{d}(u,v),\hspace{1em}u,v\in \Lambda.
\end{equation}
    
For $u \in \Lambda$, define $\mathcal{F}(u)(t)$ as
\begin{equation}
    \mathcal{F}(u)(t)=e^{t\Delta}\phi_0+\int_0^te^{(t-s)\Delta}(\mu[u]u-\beta|u|^{2}u-Vu) \, \text{d}s,
\end{equation}
obviously, $\mathcal{F}(u)(0) = \phi_0$. Combining Properties \ref{prop: 1} and \ref{prop: 2}, embedding $H^1(\Omega)\hookrightarrow L^6(\Omega)$ and H\"older's inequality, we obtain
\begin{align}
    &\|\mathcal{F}(u)\|_{L^\infty([0,T];L^2(\Omega))} \label{f1} \\
    & \le C \left( \|\phi_0\|_{L^2(\Omega)} + \| \mu[u]u-\beta|u|^{2}u-Vu \|_{L^{1}([0,T];L^2(\Omega))} \right) \nn \\
    & \le C \left( \|\phi_0\|_{L^2(\Omega)} + T\| \mu[u]u-\beta|u|^{2}u-Vu \|_{L^{\infty}([0,T];L^2(\Omega))} \right) \nn \\
    & \le C \Big(\|\phi_0\|_{L^2(\Omega)} + T (\|\mu[u]\|_{L^{\infty}[0,T]} + \|V\|_{L^{\infty}(\Omega)})\|u\|_{L^{\infty}\big([0,T];L^2(\Omega)\big)} + CT\|u\|_{L^{\infty}([0,T];H^1(\Omega))}^3  \Big) \nn \\
    & \le C \|\phi_0\|_{L^2(\Omega)} + C_3(M,N)T, \nn
\end{align}
and 
\begin{align}
     \|\nabla \mathcal{F}(u)\|_{L^\infty([0,T];L^2(\Omega))}  &\le C \left( \|\phi_0\|_{H^1(\Omega)} +  \| \mu[u]u-\beta|u|^{2}u-Vu \|_{L^{2}([0,T];L^2(\Omega))} \right) \label{f2} \\
    & \le C \left( \|\phi_0\|_{H^1(\Omega)} + T^{\frac12} \| \mu[u]u-\beta|u|^{2}u-Vu \|_{L^{\infty}([0,T];L^2(\Omega))} \right) \nn \\
    & \le  C \|\phi_0\|_{H^1(\Omega)} + C_4(M,N)T^{\frac12}. \nn
\end{align}
    
By using \eqref{mues}, \eqref{f1} and \eqref{f2}, there exist  constants $C^* > 0$ and $C_5 = C_5(M,N)>0$  such that
 \begin{equation}\label{equ: f5}
    \Vert \mathcal{F}(u)\Vert_{ L^\infty([0,T],H^1(\Omega))}\leq C^*\Vert \phi_0\Vert_{H^1(\Omega)}+C_5\Big(T+T^{\frac{1}{2}}\Big).
\end{equation}
Meanwhile, we get 
\begin{align}
    \|\mathcal{F}(u)\|_{L^2(\Omega)}& \ge \|e^{t\Delta}\phi_0\|_{L^2(\Omega)} - \left\|\int_0^te^{(t-s)\Delta}(\mu[u]u-\beta|u|^{2}u-Vu) \, \text{d}s \right\|_{L^{\infty}([0,T];L^2(\Omega))} \\
    & \ge \left(\|\phi_0\|_{L^2(\Omega)}^2 - 2\int_0^t \|\nabla e^{s\Delta}\phi_0\|_{L^2(\Omega)}^2 \ \text{d}s  \right)^{\frac12} - CT \nn \\
    &\ge  \left(\|\phi_0\|_{L^2(\Omega)}^2 - C T \|\nabla \phi_0\|_{L^2(\Omega)}^2  \right)^{\frac12} - CT, \nn
\end{align}
which implies 
\begin{align}\label{equ: f6}
    \inf_{t \in [0,T]} \|\mathcal{F}(u)\|_{L^2(\Omega)} \ge \|\phi_0\|_{L^2(\Omega)} - C_6(M,N)(T + T^{\frac12}).
\end{align}
For $u,v \in \Lambda$ 
\begin{align}
    & \|\mathcal{F}(u) - \mathcal{F}(v)\|_{L^\infty([0,T];L^2(\Omega))} \label{equ: f3}\\
    & \le C \| \mu[u]u - \mu[v]v-\beta(|u|^{2}u - |v|^{2}v)-V(u -v)\|_{L^{1}([0,T];L^2(\Omega))} \nn \\
    & \le C_7(M,N) T \text{d}(u,v), \nn
\end{align}
and 
\begin{align}
    & \|\nabla\mathcal{F}(u) - \nabla\mathcal{F}(v)\|_{L^\infty([0,T];L^2(\Omega))} \label{equ: f4}  \\
    & \le C \| \mu[u]u - \mu[v]v-\beta(|u|^{2}u - |v|^{2}v)-V(u -v)\|_{L^{2}([0,T];L^2(\Omega))} \nn \\
    & \le C_8(M,N) T^{\frac12} \text{d}(u,v). \nn
\end{align}
Combining \eqref{equ: f3} and \eqref{equ: f4}, we obtain
\begin{align}\label{equ: f7}
    \|\mathcal{F}(u) - \mathcal{F}(v)\|_{L^\infty([0,T];H^1(\Omega))} \le C_9(M,N)(T + T^{\frac12}) \text{d}(u,v).
\end{align}

First, we choose $M = 2 C^* \|\phi_0\|_{H^1(\Omega)}$ and $N = \|\phi_0\|_{L^2(\Omega)}$ in \eqref{equ: f5}, \eqref{equ: f6} and \eqref{equ: f7}, which leads to 
\begin{align}
    &\Vert \mathcal{F}(u)\Vert_{ L^\infty([0,T],H^1(\Omega))}\leq \frac{M}{2}+C_1'\Big(T+T^{\frac{1}{2}}\Big) \qquad \forall\, u \in \Lambda, \\
    &\inf_{t \in [0,T]} \|\mathcal{F}(u)\|_{L^2(\Omega)} \ge N - C_2'(T + T^{\frac12}) \qquad \forall\, u \in \Lambda, \\
    & \|\mathcal{F}(u) - \mathcal{F}(v)\|_{L^\infty([0,T];H^1(\Omega))} \le C_3'(T + T^{\frac12}) \text{d}(u,v) \qquad \forall\, u,v \in \Lambda.
\end{align}
Then, we choose $T^*$ sufficiently small such that, for $0 \le T \le T^*$
\begin{align*}
    C_1'\Big(T+T^{\frac{1}{2}}\Big)  \le \frac{M}{2}, \quad
    C_2'(T + T^{\frac12}) \le \frac{N}{2}, \quad 
    \text{and} \quad C_3'(T + T^{\frac12})  \le \frac12
\end{align*}
which leads to for all $u \in \Lambda$
\begin{align}
    \|\mathcal{F}(u)\Vert_{ L^\infty([0,T],H^1(\Omega))}\leq M, \qquad \inf_{t \in [0,T]} \|\mathcal{F}(u)\|_{L^2(\Omega)} \ge \frac{N}{2}, \label{equ: f8}
\end{align}
and for all $u,v \in \Lambda$
\begin{align}
    \|\mathcal{F}(u) - \mathcal{F}(v)\|_{L^\infty([0,T];H^1(\Omega))} \le \frac12 \text{d}(u,v). \label{equ: f9}
\end{align}
Inequalities \eqref{equ: f8} and \eqref{equ: f9} show that $\mathcal{F}(\cdot)$ is a contraction map on $\Lambda$ (the continuity of $\mathcal{F}(\cdot)$ with respect to $t$ can be easily seen from the definition of $\mathcal{F}(\cdot)$.) with $T = T^*$. By Banach Fixed-Point Theorem, there exists a unique $\phi \in \Lambda$ such that $\mathcal{F}(\phi) = \phi$. Then the local well-posedness of \eqref{equ: PDEofL2GF} directly follows from \cite[Proposition 2.2]{antonelli2024existence}.  
    
Let $\phi$ be the local solution given above. Since $\beta$ and $V$ are non-negative, we obtain from the energy diminishing property that 
\begin{equation}
    \Vert\nabla \phi(t)\Vert_{L^2(\Omega)}^2\leq2E(\phi(t))\leq2E(\phi_0), \hspace{1em}\forall\, t\in [0,T],
\end{equation}
which shows that the $H^1$-norm of $\phi(t)$ is bounded with respect to $t$. Thus we can extend the local solution to a global solution without losing uniqueness. The regularity of $\phi(t)$ directly follows from $\phi_0 \in H_0^1(\Omega) \cap H^2(\Omega)$ and \cite[Propositions 2.2 and 3.4]{antonelli2024existence}.

\underline{\textit{Part II (Convergence to the ground state})} Following \cite[Proposition 4.1, Corollary 4.2, and Theorem 4.3]{antonelli2024existence}, we are able to prove the convergence to the ground state as $t \to \infty$, under the condition of $\phi_0\geq 0$. The proof will be divided into three steps.

\textit{Step 1 (Sequential weak convergence)} The energy decay property and Proposition \ref{prop: 1} show that
\begin{align*}
    \phi_t\in L^2([0,\infty);L^2(\Omega)), \quad 
    \sup_{t>0}|\mu[\phi]|\leq C, \quad 
    \text{and} \quad \|\phi\|_{L^\infty([0,\infty);H^1(\Omega))} \le C.
\end{align*}
As a result, there exists a sequence $\{ t_n\}_{n\in N}$ satisfying $t_n\to \infty $ as $n\to \infty$, $\phi_{\infty} \in H_0^1(\Omega)$, and $\mu_{\infty} \in \R$ such that
\begin{align*}
    \begin{cases}
        \phi(t_n)\rightharpoonup \phi_{\infty} \qquad\text{in} \ H_0^1(\Omega), \\
        \phi_t(t_n) \to 0  \qquad\text{in} \ L^2(\Omega), \\
        \mu[\phi(t_n)]  \to \mu_{\infty} \qquad \text{in} \ \R.
    \end{cases}
\end{align*}
Then under the weak topology of $H^{-1}$, 
\begin{align*}
    \phi_t(t_n)& - \Delta\phi(t_n) + V\phi(t_n) + \beta|\phi(t_n)|^2\phi(t_n) - \mu[\phi(t_n)]\phi(t_n) \\
    & \rightharpoonup - \Delta\phi_{\infty} + V\phi_{\infty} + \beta|\phi_{\infty}|^2\phi_{\infty} - \mu_{\infty}\phi_{\infty}
\end{align*}
as $n \to \infty$. This implies in $H^{-1}$ sense, 
\begin{align}\label{equ: eq1 in Thm}
    - \Delta\phi_{\infty} + V\phi_{\infty} + \beta|\phi_{\infty}|^2\phi_{\infty} - \mu_{\infty}\phi_{\infty} = 0.
\end{align}
Weak convergence $\phi(t_n) \rightharpoonup \phi_{\infty}$ in $H_0^1$ implies $\|\phi_{\infty}\|_{L^2(\Omega)} = 1$. \eqref{equ: eq1 in Thm} shows that $\phi_{\infty}$ is a solution of the GP eigenvalue problem.

\textit{Step 2 (Sequential strong convergence)} Form \eqref{equ: eq1 in Thm}, we conclude that $\mu[\phi_{\infty}] = \mu_{\infty}$ and this leads to $\mu[\phi(t_n)] \to \mu[\phi_{\infty}]$ as $n \to \infty$, which is 
\begin{align*}
    \int_{\Omega} \Big( |\nabla\phi(t_n)|^2 + V|\phi(t_n)|^2 + \beta |\phi(t_n)|^4 \Big) \ \text{d}x \to \int_{\Omega} \Big( |\nabla\phi_{\infty}|^2 + V|\phi_{\infty}|^2 + \beta |\phi_{\infty}|^4 \Big) \ \text{d}x
\end{align*}
when $n \to \infty$. Note that weak convergence in $H_0^1(\Omega)$ implies strong convergence in $L^4(\Omega)$, we have 
\begin{align*}
    \int_{\Omega} |\nabla\phi(t_n)|^2  \ \text{d}x \to \int_{\Omega} |\nabla\phi_{\infty}|^2  \ \text{d}x,
\end{align*}
which together with the weak convergence, we get $\phi(t_n) \to \phi_{\infty}$ strongly in $H_0^1(\Omega)$ as $n \to \infty$.

\textit{Step 3 (Convergence to the ground state)} If the initial value satisfies $\phi_0 \ge 0$, since $V(x) \ge 0$ and $\beta \ge 0$, by the Maximum Principle, we have $\phi(t_n) \ge 0$. This implies $\phi_{\infty} \ge 0$. By \cite[Lemma 5.3]{henning2020sobolev}, we obtain $\phi_{\infty} = \phi_{\text{GS}}$. The strong convergence of $\{\phi(t_n)\}_{n \in \mathbb{N}}$ and the energy decay property show that 
\begin{equation*}
    E[\phi(t)]\searrow  E[\phi_{\text{GS}}],\qquad \text{as}\ \ t\to\infty.
\end{equation*}
To prove $\phi(t) \to \phi_{\text{GS}}$ in $H_0^1(\Omega)$ as $t \to \infty$, we only need to prove that for every sequence $\{ \tilde{t}_k \}_{k \in \mathbb{N}}$ with $\tilde{t}_k \to \infty$ as $k \to \infty$, there exists a subsequence $\{ \tilde{t}_{n_k} \}_{k \in \mathbb{N}}$ such that 
\begin{align}
    \phi(\tilde{t}_{n_k}) \to \phi_{\text{GS}} \qquad \text{when} \quad k \to \infty.
\end{align}

Boundedness of $\{ \phi(\tilde{t}_k) \}_{k \in \mathbb{N}}$ implies there exist a subsequence $\{ \phi(\tilde{t}_{n_k}) \}_{k \in \mathbb{N}}$ and $\tilde{\phi}_{\infty} \in H_0^1(\Omega)$ such that $\phi(\tilde{t}_{n_k}) \rightharpoonup \tilde{\phi}_{\infty}$ weakly in $H_0^1(\Omega)$ as $k \to \infty$. and hence $\|\tilde{\phi}_{\infty}\|_{L^2(\Omega)} = 1$. By the lower semi-continuity of $E[\cdot]$, we have 
\begin{align*}
    E[\tilde{\phi}_{\infty}] \le \liminf_{k \to \infty} E[\phi(\tilde{t}_{n_k})] = E[\phi_{\text{GS}}].
\end{align*}
From the above, we immediately get $\tilde{\phi}_{\infty} = \phi_{\text{GS}}$ or $\tilde{\phi}_{\infty} = -\phi_{\text{GS}}$. Again by the Maximum Principle, we have $\phi(\tilde{t}_{n_k}) \ge 0$, and hence $\tilde{\phi}_{\infty} \ge 0$, which is $\tilde{\phi}_{\infty} = \phi_{\text{GS}}$. Similar to \textit{Step 2}, we obtain $\phi(\tilde{t}_{n_k}) \to \phi_{\text{GS}}$ strongly in $H_0^1(\Omega)$ as $k \to \infty$, which completes the proof.
\end{proof}

\begin{remark}\label{rem: high regularity}
	Under the same assumption in Theorem \ref{thm: wellposedness of model problem}, following the regularity theory of elliptic and parabolic PDEs \cite{evans2010partial,gilbarg2001elliptic}, we have that if $V \in H^1(\Omega) \cap L^{\infty}(\Omega)$ and the initial value further satisfies $\Delta \phi_0 \in H_0^1(\Omega)$, then 
	\begin{equation*}
			\phi \in C([0,\infty); H^2(\Omega) \cap H_0^1(\Omega) ), \qquad \phi_t\in C([0,\infty);  H_0^1(\Omega) ).
		\end{equation*}
		Moreover, if $V(x) \in H^2(\Omega)$ and $\Delta \phi_0 \in H^2(\Omega) \cap H_0^1(\Omega)$, then
		\begin{equation*}
		\phi \in C^1([0,\infty); H^2(\Omega) \cap H_0^1(\Omega)), \qquad \phi_{tt} \in C([0,\infty); L^2(\Omega) ).
	\end{equation*}
\end{remark}

\subsection{Convergence rate}

In this subsection, we will discuss the convergence rate of the solution of \eqref{equ: PDEofL2GF} to the ground state as $t \to \infty$. For simplicity, we make the following assumption.
\begin{amp}\label{amp: amp1}
	The solution of \eqref{equ: PDEofL2GF} satisfies $\phi \in C^1([0,\infty); H_0^1(\Omega)) \cap C([0,\infty); H^2(\Omega))$, and $\phi(t)$ converges to $\phi_{\text{GS}}$ strongly in $H_0^1(\Omega)$ as $t \to \infty$.
\end{amp}
\begin{remark}
    By Remark \ref{rem: high regularity}, Assumption \ref{amp: amp1} is reasonable for some potential function $V$ and initial value $\phi_0$ satisfying some regularity requirements.
\end{remark}

Our discussion starts from the following lemma.
\begin{lemma}\label{lem: key lemma}
Denote $g(t) = \frac12\|\phi_t(\cdot,t) \|_{L^2(\Omega)}^2$, then for any $\epsilon \in (0, \lambda_2 - \lambda_{\mathrm{GS}})$, there exists $T_{\epsilon} > 0$ such that 
	\begin{equation}\label{equ: key inequality}
		g'(t) \le - 2(\lambda_2 - \lambda_{\text{GS}} - \epsilon) g(t) \qquad t \ge T_{\epsilon},
	\end{equation}
	where $\lambda_2$ is the second eigenvalue of the following linear eigenvalue problem: 
	\begin{equation*}
		- \Delta u + V u + \beta |\phi_{\text{GS}}|^2 u = \lambda u.
	\end{equation*}
\end{lemma}
\begin{proof}
	Note that 
\begin{equation*}
	\phi_{tt} = \Delta \phi_t - V \phi_t - 3 \beta \phi^2 \phi_t + \left( \frac{\text{d}}{\text{d}t} \mu[\phi] \right) \phi + \mu[\phi]\phi_t,
\end{equation*}
where $\partial_{tt}\phi$ is understood in $H^{-1}$ sense. We have that 
\begin{align}
	g'(t) & = \langle\phi_t, \phi_{tt}\rangle_{H^{-1}(\Omega) \times H_0^1(\Omega)} \nn \\
	& = -(\nabla \phi_t,  \nabla \phi_t)_{L^2(\Omega)} - (V\phi_t,\phi_t)_{L^2(\Omega)}  - \beta (\phi^2 \phi_t, \phi_t)_{L^2(\Omega)} \underbrace{- 2\beta (\phi^2 \phi_t, \phi_t)_{L^2(\Omega)}}_{\le 0} \nn \\
	& \quad + \left( \frac{\text{d}}{\text{d}t} \mu[\phi] \right) \underbrace{(\phi, \phi_t)_{L^2(\Omega)}}_{ = 0} + \mu[\phi] (\phi_t,\phi_t)_{L^2(\Omega)} \nn \\
	& \le - \int_{\Omega} \Big( |\nabla \phi_t|^2 + V|\phi_t|^2 + \beta|\phi|^2 |\phi_t|^2 \Big) \ \text{d}x + \mu[\phi]\|\phi_t\|_{L^2(\Omega)}^2.
\end{align}

Denote 
\begin{equation*}
a_{\phi}(u,v):= \int_{\Omega} \Big( \nabla u \cdot \nabla v + V uv + \beta|\phi|^2 uv \Big) \ \text{d}x,
\end{equation*}
\begin{equation*}
	C_{\text{inf}}:= \liminf_{t \to \infty} \frac{a_{\phi}(\phi_t,\phi_t)}{\|\phi_t\|_{L^2(\Omega)}^2}.
\end{equation*}
Assume that $(t_n)_{n \in \mathbb{N}}$ with $t_n \to \infty$ satisfies
\begin{equation*}
	\lim_{n \to \infty} \frac{a_{\phi(t_n)}(\phi_t(t_n),\phi_t(t_n))}{\|\phi_t(t_n)\|_{L^2(\Omega)}^2} = C_{\text{inf}}.
\end{equation*}
Let
\begin{equation*}
	z_n:= \frac{\phi_t(t_n)}{\|\phi_t(t_n)\|_{L^2(\Omega)}},
\end{equation*}
we see that $\{z_n\}_{n \in \mathbb{N}}$ is a bounded sequence in $H_0^1(\Omega)$ and
\begin{align*}
	&\left| \int_{\Omega} |\phi(t_n)|^2 z_n^2 \ \text{d}x - \int_{\Omega} |\phi_{\text{GS}}|^2 z_n^2 \ \text{d}x    \right| \\
	 & \le \| \phi(t_n) + \phi_{\text{GS}}\|_{L^4(\Omega)}\| \phi(t_n) - \phi_{\text{GS}}\|_{L^4(\Omega)}\|z_n^2\|_{L^2(\Omega)}  \to 0.
\end{align*}
Hence, 
\begin{equation*}
	\lim_{n \to \infty} \underbrace{ \int_{\Omega} \Big( |\nabla  z_n|^2 + V| z_n|^2 + \beta|\phi_{\text{GS}}|^2 |z_n|^2 \Big) \ \text{d}x}_{=: a_{\phi_{\text{GS}}}(z_n,z_n)} = C_{\text{inf}}.
\end{equation*}

Since $\{z_n\}_{n \in \mathbb{N}}$ is bounded in $H^1(\Omega)$, there exists a weak limit 
\begin{equation*}
	z_n \rightharpoonup \hat{z} \qquad \text{weakly in $H^1(\Omega)$}.
\end{equation*}
This indicates $\|\hat{z}\|_{L^2(\Omega)} = 1$ and
\begin{align*}
	&\left| (\hat{z}, \phi_{\text{GS}})_{L^2(\Omega)} - (z_n, \phi(t_n))_{L^2(\Omega)}  \right| \\
	 &\le \|\hat{z} - z_n\|_{L^2(\Omega)}\|\phi_{\text{GS}}\|_{L^2(\Omega)} + \|z_n\|_{L^2(\Omega)}\|\phi_{\text{GS}} - \phi(t_n)\|_{L^2(\Omega)}  \to 0,
\end{align*} 
which leads to
\begin{equation*}
	(\hat{z}, \phi_{\text{GS}})_{L^2(\Omega)} = 0.
\end{equation*}
Lower semicontinuity of weakly converging sequence implies 
\begin{equation*}
	C_{\text{inf}} = \lim_{n \to \infty}a_{\phi_{\text{GS}}}(z_n,z_n) \ge a_{\phi_{\text{GS}}}(\hat{z}, \hat{z}) \ge \inf_{v \in \text{span}\{\phi_{\text{GS}}\}^\bot} \frac{a_{\phi_{\text{GS}}}(v, v)}{\|v\|_{L^2(\Omega)}^2}.
\end{equation*}

With the Courant-Fischer theorem we get
\begin{equation*}
	C_{\text{inf}} \ge \lambda_2,
\end{equation*}
where $\lambda_2$ is the second eigenvalue of linear eigenvalue problem: Find $(u,\lambda) \in H_0^1(\Omega) \times \R$ such that
\begin{equation*}
	- \Delta u + V u + \beta |\phi_{\text{GS}}|^2 u = \lambda u.
\end{equation*}
	
With the fact $\mu[\phi(t)] \to \lambda_{\text{GS}}$ as $t \to \infty$, there exists $T_{\epsilon} > 0$ such that for $t \ge T_{\epsilon}$
\begin{equation}\label{equ: 1}
	\mu[\phi(t)] \le \lambda_{\text{GS}} + \frac{\epsilon}{2}
\end{equation}
and 
\begin{equation}\label{equ: 2}
	\frac{a_{\phi}(\phi_t,\phi_t)}{\|\phi_t\|_{L^2(\Omega)}^2} \ge \lambda_2 - \frac{\epsilon}{2}.
\end{equation}

Combining \eqref{equ: 1} and \eqref{equ: 2}, we have that for $t \ge T_{\epsilon}$, 
\begin{align*}
	g'(t) & \le - \int_{\Omega} \Big( |\nabla \phi_t|^2 + V|\phi_t|^2 + \beta|\phi|^2 |\phi_t|^2 \Big) \ \text{d}x + \mu[\phi]\|\phi_t\|_{L^2(\Omega)}^2  = \|\phi_t\|_{L^2(\Omega)}^2 \left( \mu[\phi] - \frac{a_{\phi}(\phi_t,\phi_t)}{\|\phi_t\|_{L^2(\Omega)}^2} \right) \\
	& \le - (\lambda_2 - \lambda_{\text{GS}} - \epsilon) \|\phi_t\|_{L^2(\Omega)}^2 = -2(\lambda_2 - \lambda_{\text{GS}} - \epsilon)g(t).
\end{align*}
This completes the proof.
\end{proof}

By Gr\"onwall's inequality, we arrive at
\begin{coro}\label{coro: 1}
Let
 $\epsilon\in(0,\lambda_2 - \lambda_{GS})$. If $t \ge T_{\epsilon}$, then 
	\begin{equation}\label{equ: key inequality 1}
		g(t) \le C_{\epsilon} \exp\big(-2(\lambda_2 - \lambda_{\text{GS}} - \epsilon) t \big).
	\end{equation}	
\end{coro}

With the inequality \eqref{equ: key inequality 1}, we obtain the following exponential convergence.

\begin{theorem}\label{thm: exponential covergence}
	Let $\epsilon\in(0,\lambda_2 - \lambda_{GS}).$ If $t\ge T_\epsilon$, then
	\begin{itemize}
		\item exponential convergence of energy 
		\begin{equation}\label{equ: exponential convergence of energy}
		E(\phi(t)) - E(\phi_{\text{GS}}) \le C_{\epsilon} \exp\big(-2(\lambda_2 - \lambda_{\text{GS}} - \epsilon) t \big),
	\end{equation}
	\item exponential convergence to ground state 
	\begin{equation}\label{equ: exponential convergence to ground state}
		\|\phi(t) - \phi_{\text{GS}}\|_{L^2(\Omega)} \le C_{\epsilon} \exp\big(-(\lambda_2 - \lambda_{\text{GS}} - \epsilon) t \big),
	\end{equation}
	\item exponential convergence of eigenvalue
	\begin{equation}\label{equ: exponential convergence of eigenvalue}
		\left|\mu[\phi(t)] - \lambda_{\text{GS}}\right| \le C_{\epsilon} \exp\big(-(\lambda_2 - \lambda_{\text{GS}} - \epsilon) t \big).
	\end{equation}
	\end{itemize}
\end{theorem}
\begin{proof}
	Note that 
	\begin{equation*}
		E(\phi(t)) - E(\phi_{\text{GS}}) = \int_{t}^{\infty} 2g(t) \ \text{d}t,
	\end{equation*}
	then \eqref{equ: exponential convergence of energy} follows from \eqref{equ: key inequality 1}. \eqref{equ: exponential convergence to ground state} can be similarly derived by \eqref{equ: key inequality 1} and 
	\begin{equation*}
		\|\phi(t) - \phi_{\text{GS}}\|_{L^2(\Omega)} \le  \int_{t}^{\infty} \|\phi_t\|_{L^2(\Omega)} \ \text{d}t \le \int_{t}^{\infty} \sqrt{2 g(t)} \ \text{d}t. 
	\end{equation*} 
	
	With \eqref{equ: exponential convergence of energy} and \eqref{equ: exponential convergence to ground state}, we have 
	\begin{align*}
		|\mu[\phi(t)] - \lambda_{\text{GS}}| & \le 2(E(\phi(t)) - E(\phi_{\text{GS}})) + \left|\frac{\beta}{2} \int_{\Omega} \Big( |\phi(t)|^4 - |\phi_{\text{GS}}|^4 \Big) \ \text{d}x \right| \\
		& \le C_{\epsilon}\exp\big(-2(\lambda_2 - \lambda_{\text{GS}} - \epsilon) t \big) 
 \\
 & \quad + C \|\phi - \phi_{\text{GS}}\|_{L^2(\Omega)}\|\phi + \phi_{\text{GS}}\|_{L^6(\Omega)} \|\phi^2 + \phi_{\text{GS}}^2\|_{L^3(\Omega)} \\
		& \le C_{\epsilon}\exp\big(-(\lambda_2 - \lambda_{\text{GS}} - \epsilon) t \big).
	\end{align*}
This completes the proof of \eqref{equ: exponential convergence of eigenvalue}.
\end{proof}

In addition, we further obtain the strong convergence in $H^2(\Omega)$.
\begin{theorem}
	$\phi(t)$ strongly converges to $\phi_{\text{GS}}$ in $H^2(\Omega)$ as $t \to \infty$, and hence $\|\phi(t)\|_{H^2(\Omega)}$ is uniformly bounded.
\end{theorem}
\begin{proof}
	We only need to prove the first claim. By Corollary \ref{coro: 1}, we have $\phi_t \to 0$ in $L^2(\Omega)$, which implies 
	\begin{align*}
		\Delta \phi(t) & = \phi_t + V \phi + \beta \phi^3 - \mu[\phi] \phi \\
		& \to V \phi_{\text{GS}} + \beta \phi_{\text{GS}}^3 - \lambda_{\text{GS}} \phi_{\text{GS}} = \Delta \phi_{\text{GS}} \qquad \text{in $L^2(\Omega)$}.
	\end{align*}
	Namely, $\phi(t) \to \phi_{\text{GS}}$ in $H^2(\Omega) \cap H_0^1(\Omega)$.
\end{proof}

\section{Numerical analysis}\label{sec: Numerical analysis}

In this section, we introduce the numerical discretization scheme of the $L^2$ normalized gradient flow corresponding to the GP eigenvalue problem and its convergence analysis. Combining a backward-forward Euler discretiaztion in time and a linear finite element discretiaztion in space, we propose a fully discrete scheme for the model problem, and give local in time convergence results with respect to the temporal mesh $\tau$ and spatial mesh $h$. 

\subsection{Numerical scheme}

We consider the numerical approximation of \eqref{equ: PDEofL2GF} in finite time $T$. Taking an inner product with $v \in H_0^1(\Omega)$ in \eqref{equ: PDEofL2GF}, we obtain the variational from of \eqref{equ: PDEofL2GF} as follows:
\begin{align}\label{equ: variational form}
    (\phi_t,v)_{L^2(\Omega)} + (\nabla\phi,v )_{L^2(\Omega)}   = - (V\phi,v)_{L^2(\Omega)} - \beta (\phi^3,v)_{L^2(\Omega)} + \mu[\phi](\phi,v)_{L^2(\Omega)} \qquad \forall\,  t > 0.
\end{align}

We first consider the temporal semi-discretization. Let $0 = t_0 < t_1 < \cdots < t_N = T$ denote a uniform partition of the time interval $[0,T]$ with stepsize $\tau = T/N$. Let $\phi^n$ be the approximation of $\phi(t_n)$. By applying the backward-forward Euler discretiaztion to \eqref{equ: variational form} at $t = t_{n+1}$, we obtain
\begin{align}
    & \left( \frac{\tilde{\phi}^{n+1} - \phi^n}{\tau}, v \right)_{L^2(\Omega)} + (\nabla \tilde{\phi}^{n+1}, \nabla v)_{L^2(\Omega)} \nn \\
    & \quad = - (V\phi^n,v)_{L^2(\Omega)} - \beta((\phi^n)^3,v)_{L^2(\Omega)} + \mu[\phi^n](\phi^n,v)_{L^2(\Omega)} \quad \forall v \in H_0^1(\Omega). \label{equ: temporal semi-discretization}
\end{align}
Here $\tilde{\phi}^{n+1}$ is an temporary approximation of $\phi(t_{n+1})$. In general, $\|\tilde{\phi}^{n+1}\|_{L^2(\Omega)} \neq 1$. To get a mass conservation approximation, we normalize $\tilde{\phi}^{n+1}$ and set
\begin{align*}
    \phi^{n+1}   = \frac{ \tilde{\phi}^{n+1}  }{\left\| \tilde{\phi}^{n+1}  \right\|_{L^2(\Omega)}}.
\end{align*}

We then use the finite element method to do the spatial discretization of \eqref{equ: temporal semi-discretization}. Let $\mathcal{T}_h$ be a shape regular and quasi-uniform triangulation of $\Omega$ with mesh size $h$, and denote $S_h \subset H_0^1(\Omega)$ the piecewise linear finite element space corresponding to $\mathcal{T}_h$. The Ritz projection operator $R_h:H_0^1(\Omega) \to S_h$ and the discrete Laplacian operator $\Delta_h: S_h \to S_h$ are defined as follows:
\begin{align*}
	& (\nabla R_h u, \nabla v_h)_{L^2(\Omega)} = (\nabla  u, \nabla v_h)_{L^2(\Omega)} \qquad \forall\, u \in H_0^1(\Omega), \, \forall\, v_h \in S_h, \\
	& (\Delta_h u_h, v_h)_{L^2(\Omega)} = - (\nabla  u_h, \nabla v_h)_{L^2(\Omega)} \qquad \forall\, u_h, \, v_h \in S_h.
\end{align*}

Combing the finite element discretization in space, we apply the following full-discrete scheme to \eqref{equ: PDEofL2GF}. 

For $0 \le n \le N-1$ and given $\phi_h^n \in S_h$, we first obtain $\tilde{\phi}_h^{n+1}$ by solving the following equation: Find $\tilde{\phi}_h^{n+1} \in S_h$, such that,
\begin{align}
	&\left(\frac{ \tilde{\phi}_h^{n+1} -  \phi_h^{n}}{\tau},v_h\right)_{L^2(\Omega)} + (\nabla \tilde{\phi}_h^{n+1}, \nabla v_h )_{L^2(\Omega)} \nn \\
    &= \left( - V\phi_h^{n}, v_h \right)_{L^2(\Omega)} - \beta \left( |\phi_h^{n}|^2\phi_h^{n}, v_h \right)_{L^2(\Omega)} - \mu[\phi_h^{n}]\left(\phi_h^{n},v_h \right)_{L^2(\Omega)} \qquad \forall\, v_h \in S_h. \label{equ: numerical scheme 1} 
    \end{align}
Then, we get $\phi_h^{n+1}$ by normalizing $\tilde{\phi}_h^{n+1}$, i.e.,
\begin{align}
	 \phi_h^{n+1}   = \frac{ \tilde{\phi}_h^{n+1}  }{\left\| \tilde{\phi}_h^{n+1}  \right\|_{L^2(\Omega)}}. \label{equ: numerical scheme 2}
\end{align}
The initial value $\phi_h^0$ is given by 
\begin{equation}\label{equ: numerical scheme 3}
	 \phi_h^0  =  \frac{R_h  \phi_0 }{\| R_h  \phi_0 \|_{L^2(\Omega)} }.
\end{equation}

The subsequent theorem shows the well-posedness of the numerical scheme \eqref{equ: numerical scheme 1}--\eqref{equ: numerical scheme 3}.
\begin{theorem}
	The numerical scheme \eqref{equ: numerical scheme 1}--\eqref{equ: numerical scheme 3} is well-posed for any $\tau >0$ and $h > 0$.
\end{theorem}
\begin{proof}
The numerical scheme \eqref{equ: numerical scheme 1}--\eqref{equ: numerical scheme 3} contains two steps, the first step is to get $\tilde{\phi}_h^{n+1}$ by solving the following equation
\begin{align*}
    \tau(\nabla \tilde{\phi}_h^{n+1}, \nabla v_h)_{L^2(\Omega)} + ( \tilde{\phi}_h^{n+1}, v_h)_{L^2(\Omega)} = (f^n,v_h)_{L^2(\Omega)} \qquad \forall\, v_h \in S_h,
\end{align*}
here $f^n \in L^2(\Omega)$ only depends on $\phi_h^n$. By noticing that
\begin{align*}
    \tau(\nabla u_h, \nabla u_h)_{L^2(\Omega)} + ( u_h, u_h)_{L^2(\Omega)} \ge \min\{ \tau,1 \}\|u_h\|_{H^1(\Omega)}^2 \qquad \forall u_h \in S_h,
\end{align*}
then the well-posedness of $\tilde{\phi}_h^{n+1}$ can be obtained by the Lax-Milgram lemma \cite{brenner2008mathematical}.
The second step is to get $\phi^{n+1}_h$ by normalizing $\tilde{\phi}_h^{n+1}$, to show the well-posedness of the second step, we only need to prove that for any $0 \le n \le N-1$  
	\begin{equation*}
		\left\| \tilde{\phi}_h^{n+1}  \right\|_{L^2(\Omega)} \neq 0.
	\end{equation*}
	If not, suppose 
	\begin{equation*}
		\left\| \tilde{\phi}_h^{n+1}  \right\|_{L^2(\Omega)} = 0
	\end{equation*}
	for some $0 \le n \le N-1$, then equation
	\begin{equation}\label{equ: 4.13}
		\left( -\frac{     \phi_h^{n} }{\tau}, v_h \right)_{L^2(\Omega)} = \left( - V \phi_h^{n}  - \beta|\phi_h^{n}|^2  \phi_h^{n}  + \mu[\phi_h^{n}]  \phi_h^{n} ,v_h \right)_{L^2(\Omega)} 
	\end{equation}
    holds for any $v_h \in S_h$. Choose $ v_h =  \phi_h^{n} $ in \eqref{equ: 4.13}, we can get 
	\begin{equation*}
	 \underbrace{- \frac{1}{\tau}}_{<0} = \underbrace{\int_{\Omega} |\nabla \phi_h^n|^2 \ \text{d}x}_{\ge 0},
	\end{equation*}
	which is a contradiction.
\end{proof}

\subsection{Convergence}

In this subsection, we will show the convergence of the numerical scheme \eqref{equ: numerical scheme 1}--\eqref{equ: numerical scheme 3}. It is well known that error estimates with respect to the mesh size depend on the regularity of the solution. In our analysis, we require the following assumption.
\begin{amp}
	The potential function $V \in L^2(\Omega)$, and the solution of \eqref{equ: PDEofL2GF} satisfies 
	\begin{equation}
		\phi \in C^1([0,\infty); H^2(\Omega) \cap H_0^1(\Omega)), \qquad \phi_{tt}  \in C([0,\infty); L^2(\Omega) ).
	\end{equation}
\end{amp}
\begin{remark}
    Although the mathematical analysis of the $L^2$ normalized gradient flow is considered under $V \in L^{\infty}(\Omega)$, our numerical analysis will show that the convergence also works under a weaker potential function $V \in L^2(\Omega)$ which includes the Coulomb potential.
\end{remark}

\subsubsection{Technical lemmas}

In this part, we will list several technical lemmas which will be used in the convergence analysis, whose proofs are provided in Appendix \ref{sec: appendix}. In order to simplify the notation, for $0 \neq v \in L^2(\Omega)$, denote 
\begin{equation*}
	\widehat{v} = \frac{v}{\|v\|_{L^2(\Omega)}}.
\end{equation*}
	
\begin{lemma}\label{lem: technical lemma 1}
	For $\|u\|_{H^1(\Omega)} \le M$, $\|v\|_{H^1(\Omega)} \le M$, $\|u\|_{L^2(\Omega)} = \|v\|_{L^2(\Omega)} =1$, there holds
	\begin{equation*}
		|\mu[u] - \mu[v]| \le C_M \|u - v\|_{H^1(\Omega)}.
	\end{equation*}
\end{lemma}

\begin{lemma}\label{lem: technical lemma 2}
	If $u \in H_0^1(\Omega) \cap H^2(\Omega)$ with $\|u\|_{L^2(\Omega)} = 1$, then 
	\begin{equation*}
		\left| \mu[u] - \mu[\widehat{R_hu}]  \right| \le Ch^2,
	\end{equation*}
    where the constant $C$ only depends on $\|u\|_{H^2(\Omega)}$.
\end{lemma}

\begin{lemma}[discrete Gronwall's inequality \cite{heywood1990finite}]
Let $\tau$, $B$ and $a_k$, $b_k$, $c_k$, $\gamma_k$ be nonnegative numbers such that
\begin{equation*}
	a_n + \tau \sum_{k = 0}^n \gamma_k b_k \le \tau \sum_{k = 0}^n \gamma_k a_k + \tau \sum_{k = 0}^n c_k + B.
	\end{equation*}
	Suppose that $\tau\gamma_k < 1$ for all $k$, and set $\sigma_k = ( 1- \tau\gamma_k )^{-1}$. Then 
	\begin{equation*}
	a_n + \tau \sum_{k = 0}^n \gamma_k b_k \le \exp\left( \tau \sum_{k =0}^n \gamma_k \sigma_k \right)\left( \tau\sum_{k = 0}^n c_k + B \right).
 \end{equation*}
\end{lemma}

\subsubsection{Consistency}

Note that the exact solution satisfies: for any $v_h \in S_h$
\begin{align*}
	&(\phi_t(t_{n+1}), v_h)_{L^2(\Omega)} + (\nabla \phi(t_{n+1}), \nabla v_h)_{L^2(\Omega)}\\
   & \qquad\qquad\qquad = -(V\phi(t_{n+1}), v_h)_{L^2(\Omega)} - (\beta\phi(t_{n+1})^3,v_h)_{L^2(\Omega)} + \mu[\phi(t_{n+1})](\phi(t_{n+1}),v_h)_{L^2(\Omega)},
\end{align*}
which implies 
\begin{align*}
	&\left( \frac{\widehat{R_h\phi(t_{n+1})} - \widehat{R_h\phi(t_{n})}}{\tau}, v_h    \right)_{L^2(\Omega)} + (\nabla \widehat{R_h\phi(t_{n+1})}, \nabla v_h)_{L^2(\Omega)} \\
	 =& -(V\widehat{R_h\phi(t_{n})}, v_h)_{L^2(\Omega)} - (\beta (\widehat{R_h\phi(t_{n})})^3, v_h)_{L^2(\Omega)} \\
	 & + \mu[\widehat{R_h\phi(t_{n})}] (\widehat{R_h\phi(t_{n})},v_h)_{L^2(\Omega)} + \mathcal{E}(v_h).
\end{align*}
Here, $\mathcal{E}(v_h)$ denotes the truncation error, given by
\begin{align}
	\mathcal{E}(v_h) & = \left( \frac{\widehat{R_h\phi(t_{n+1})} - \widehat{R_h\phi(t_{n})}}{\tau}, v_h    \right)_{L^2(\Omega)} - (\phi_t(t_{n+1}), v_h)_{L^2(\Omega)} \nn \\
	& \quad +  \underbrace{\frac{1}{\|R_h\phi(t_{n+1})\|_{L^2(\Omega)}}(\nabla R_h \phi(t_{n+1}) - \nabla \phi(t_{n+1}), \nabla v_h)_{L^2(\Omega)}}_{=0} \nn \\
	& \quad + \left( \frac{1}{\|R_h\phi(t_{n+1})\|_{L^2(\Omega)}} - 1 \right)(\nabla \phi(t_{n+1}), \nabla v_h)_{L^2(\Omega)} \nn \\
	&\quad + (V\widehat{R_h\phi(t_{n})}, v_h)_{L^2(\Omega)}  - (V\phi(t_{n+1}), v_h)_{L^2(\Omega)} \nn \\
	&\quad + (\beta (\widehat{R_h\phi(t_{n})})^3, v_h)_{L^2(\Omega)} - (\beta \phi(t_{n+1})^3, v_h)_{L^2(\Omega)} \nn \\
	 &\quad + \mu[\phi(t_{n+1})] (\phi(t_{n+1}),v_h)_{L^2(\Omega)} - \mu[\widehat{R_h\phi(t_{n})}] (\widehat{R_h\phi(t_{n})},v_h)_{L^2(\Omega)} \nn \\
	 & =: \mathcal{E}_1(v_h) + \mathcal{E}_2(v_h) + \mathcal{E}_3(v_h) +\mathcal{E}_4(v_h) + \mathcal{E}_5(v_h).
\end{align}

The rest of this part is to prove the following estimate of the truncation error $\mathcal{E}(\cdot)$.
\begin{lemma}\label{lem: truncation error}
	The truncation error of the numerical scheme \eqref{equ: numerical scheme 1}--\eqref{equ: numerical scheme 3} is given by
	\begin{equation}
		|\mathcal{E}(v_h)| \le C  (\tau + h^2) \|\nabla v_h\|_{L^2(\Omega)} \qquad \forall v_h \in S_h.
	\end{equation}
\end{lemma}
\begin{proof}
We will estimate for $\mathcal{E}_j(v_h)$, $j = 1,2,3,4,5$, respectively.
First, we have 
\begin{align}
	|\mathcal{E}_1(v_h)| & \le \left\| \frac{\widehat{R_h\phi(t_{n+1})} - \widehat{R_h\phi(t_{n})}}{\tau} - \frac{R_h\phi(t_{n+1}) - R_h\phi(t_n)}{\tau} \right\|_{L^2(\Omega)}\|v_h\|_{L^2(\Omega)} \nn  \\
	& \quad + \left\| \frac{R_h\phi(t_{n+1}) - R_h\phi(t_n)}{\tau} - \frac{\phi(t_{n+1}) - \phi(t_n)}{\tau}  \right\|_{L^2(\Omega)}\|v_h\|_{L^2(\Omega)} \nn \\
	& \quad + \left\| \frac{\phi(t_{n+1}) - \phi(t_n)}{\tau} - \phi_t(t_{n+1})  \right\|_{L^2(\Omega)}\|v_h\|_{L^2(\Omega)} \nn \\
	& \le I_1\|v_h\|_{L^2(\Omega)} + Ch^2\|v_h\|_{L^2(\Omega)}   + Ch^2  \left\|\frac{\phi(t_{n+1}) - \phi(t_n)}{\tau}  \right\|_{H^2(\Omega)}\|v_h\|_{L^2(\Omega)} + C \tau \|v_h\|_{L^2(\Omega)} \nn \\
	& \le I_1 \|v_h\|_{L^2(\Omega)} + Ch^2 \|v_h\|_{L^2(\Omega)} + C \tau \|v_h\|_{L^2(\Omega)},
\end{align}
where 
\begin{align*}
	 I_1 = \left\| \frac{(\phi(t_{n+1})/\|R_h\phi(t_{n+1})\|_{L^2(\Omega)} - \phi(t_{n+1})) - (\phi(t_n)/\|R_h\phi(t_n)\|_{L^2(\Omega)} - \phi(t_n))}{\tau}  \right\|_{L^2(\Omega)}.
\end{align*}
It is seen that
\begin{align*}
	I_1 \le \max_{t_n \le t \le t_{n+1}}\Big\| \partial_t\big(\phi(t)/\|R_h\phi(t)\|_{L^2(\Omega)} - \phi(t)\big) \Big\|_{L^2(\Omega)}.
\end{align*}
Since
\begin{align*}
	&\partial_t(\phi(t)/\|R_h\phi(t)\|_{L^2(\Omega)} - \phi(t)) \\
	  = &\partial_t \phi(t) (1/\|R_h\phi(t)\|_{L^2(\Omega)} - 1) + \phi \left( \frac{1}{\|R_h\phi(t)\|_{L^2(\Omega)}} - \frac{1}{\|\phi(t)\|_{L^2(\Omega)}}  \right)' \\
	 =& \phi_t(t) (1/\|R_h\phi(t)\|_{L^2(\Omega)} - 1) + \phi\left( \frac{(R_h\phi(t), R_h\phi_t(t))_{L^2(\Omega)}}{\|R_h\phi(t)\|_{L^2(\Omega)}^3} - \frac{(\phi(t), \phi_t(t))_{L^2(\Omega)}}{\|\phi(t)\|_{L^2(\Omega)}^3} \right),
\end{align*}
we are able to estimate as follows
\begin{align*}
    \| \phi_t(t) (1/\|R_h\phi(t)\|_{L^2(\Omega)} - 1)\|_{L^2(\Omega)} \le C \|\phi(t) - R_h\phi(t)\|_{L^2(\Omega)} \le C h^2,
\end{align*}
and
\begin{align*}
	& \left| \frac{(R_h\phi(t), R_h\phi_t(t))_{L^2(\Omega)}}{\|R_h\phi(t)\|_{L^2(\Omega)}^3} - \frac{(\phi(t), \phi_t(t))_{L^2(\Omega)}}{\|\phi(t)\|_{L^2(\Omega)}^3} \right| \\
	   \le & \left| \frac{(R_h\phi(t), R_h\phi_t(t))_{L^2(\Omega)}}{\|R_h\phi(t)\|_{L^2(\Omega)}^3} - \frac{(\phi(t), \phi_t(t))_{L^2(\Omega)}}{\|R_h\phi(t)\|_{L^2(\Omega)}^3} \right| \\
	    & + \left| \frac{(\phi(t), \phi_t(t))_{L^2(\Omega)}}{\|R_h\phi(t)\|_{L^2(\Omega)}^3} - \frac{(\phi(t), \phi_t(t))_{L^2(\Omega)}}{\|\phi(t)\|_{L^2(\Omega)}^3} \right| \\
	   \le & C \left| (R_h\phi(t), R_h\phi_t(t))_{L^2(\Omega)} -  (\phi(t), \phi_t(t))_{L^2(\Omega)} \right| + C h^2 \\
	   \le & C \left| (R_h\phi(t), R_h\phi_t(t) - \phi_t(t))_{L^2(\Omega)}  \right| + C \left| (R_h\phi(t) - \phi(t), \phi_t)_{L^2(\Omega)}  \right| +Ch^2 \\
	   \le & C h^2\|\phi_t\|_{H^2(\Omega)} + Ch^2 \|\phi\|_{H^2(\Omega)} + Ch^2,
\end{align*}
we get $I_1 \le C h^2$, which implies 
\begin{equation}
	|\mathcal{E}_1(v_h)| \le C (\tau + h^2)\|v_h\|_{L^2(\Omega)}.
\end{equation}

Second, it is easy to show that 
\begin{align}
 |\mathcal{E}_2(v_h)| &=\left| \left( \frac{1}{\|R_h\phi(t_{n+1})\|_{L^2(\Omega)}} - 1 \right)(\nabla \phi(t_{n+1}), \nabla v_h)_{L^2(\Omega)} \right| \nn \\
   & \le \left| \frac{1}{\|R_h\phi(t_{n+1})\|_{L^2(\Omega)}} - 1 \right|\|\nabla \phi(t_{n+1})\|_{L^2(\Omega)} \|\nabla v_h\|_{L^2(\Omega)} \le C h^2\|v_h\|_{H^1(\Omega)}.
\end{align} 

Third, by Sololev embedding inequality, there holds
\begin{align}
	|\mathcal{E}_3(v_h)| & = \left|(V\widehat{R_h\phi(t_{n})}, v_h)_{L^2(\Omega)}  - (V\phi(t_{n+1}), v_h)_{L^2(\Omega)} \right| \nn \\
    & \le \left|(V\widehat{R_h\phi(t_{n})} - V\widehat{R_h\phi(t_{n+1})}, v_h )_{L^2(\Omega)}\right|  + \left|(V\widehat{R_h\phi(t_{n+1})} - V\phi(t_{n+1}), v_h )_{L^2(\Omega)}\right| \nn \\
	& \le \|V\|_{L^2(\Omega)}\left\|\widehat{R_h\phi(t_{n+1})} - \widehat{R_h\phi(t_{n})}\right\|_{L^3(\Omega)}\|v_h\|_{L^6(\Omega)} \nn  \\
	& \quad + \|V\|_{L^2(\Omega)}\|\widehat{R_h\phi(t_{n+1})} - \phi(t_{n+1})\|_{L^3(\Omega)}\|v_h\|_{L^6(\Omega)} \nn \\
	& \le C \tau \|v_h\|_{H^1(\Omega)} + C h^2\|v_h\|_{H^1(\Omega)}.
\end{align}

Fourth, 
\begin{align}
	|\mathcal{E}_4(v_h)| & = \left|(\beta (\widehat{R_h\phi(t_{n})})^3, v_h)_{L^2(\Omega)} - (\beta \phi(t_{n+1})^3, v_h)_{L^2(\Omega)} \right|\nn \\ & \le |(\beta (\widehat{R_h\phi(t_{n})})^3-\beta (\widehat{R_h\phi(t_{n+1})})^3, v_h)_{L^2(\Omega)}| \nn \\
	& \quad + |(\beta (\widehat{R_h\phi(t_{n+1})} )^3-\beta \phi(t_{n+1})^3, v_h)_{L^2(\Omega)}| \nn \\
	& \le C \tau \|v_h\|_{L^2(\Omega)} + C h^2\|v_h\|_{H^1(\Omega)}.
\end{align}

Finally, by Lemmas \ref{lem: technical lemma 1} and \ref{lem: technical lemma 2}, we can estimate as follows
\begin{align}
	|\mathcal{E}_5(v_h)|  & = \left| \mu[\phi(t_{n+1})] (\phi(t_{n+1}),v_h)_{L^2(\Omega)} - \mu[\widehat{R_h\phi(t_{n})}] (\widehat{R_h\phi(t_{n})},v_h)_{L^2(\Omega)} \right| \nn \\ & \le |\mu[\phi(t_{n+1})] - \mu[\phi(t_{n})]||(\phi(t_{n+1}),v_h)_{L^2(\Omega)}| \nn \\
	& \quad + | \mu[\phi(t_{n})]||(\phi(t_{n+1}) - \phi(t_n),v_h)_{L^2(\Omega)}| \nn \\
	& \quad + |\mu[\phi(t_{n})] - \mu[\widehat{R_h\phi(t_{n})}]||(\phi(t_{n}),v_h)_{L^2(\Omega)}| \nn \\
	& \quad + |\mu[\widehat{R_h\phi(t_{n})}]||(\phi(t_{n}) - \widehat{R_h\phi(t_{n})},v_h)_{L^2(\Omega)}| \nn \\
	& \le C \|\phi(t_{n+1}) - \phi(t_{n})\|_{H^1(\Omega)}\|v_h\|_{L^2(\Omega)} + C\tau\|v_h\|_{L^2(\Omega)} \nn \\
	& \quad  + C h^2 \|v_h\|_{L^2(\Omega)} + C h^2\|v_h\|_{L^2(\Omega)} \nn \\
	& \le C \tau \|v_h\|_{L^2(\Omega)} + Ch^2\|v_h\|_{L^2(\Omega)}.
\end{align}

Combining estimates of $\mathcal{E}_1$--$\mathcal{E}_5$, we arrive at the given conclusion.
\end{proof}

\subsubsection{Error equation}

We define the following two types of error functions:
\begin{equation*}
	e_h^n = \widehat{R_h\phi(t_{n})} - \phi_h^{n}
\end{equation*} 
for $0 \le n \le N$, and
\begin{equation*}
	\tilde{e}^n_h = \widehat{R_h\phi(t_{n})} - \tilde{\phi}_h^n
\end{equation*}
for $1 \le n \le N $, with $\tilde{e}_h^0 = 0$. Then the error equation can be written as: For $0 \le n \le N-1$ 
\begin{align}
	& \left( \frac{\tilde{e}_h^{n+1} - e_h^n}{\tau}, v_h \right)_{L^2(\Omega)} + (\nabla \tilde{e}_h^{n+1}, \nabla v_h)_{L^2(\Omega)} \nn \\
	= & - (V e_h^n, v_h)_{L^2(\Omega)} - (\beta (\widehat{R_h\phi(t_{n})})^3 - \beta (\phi_h^n)^3, v_h)_{L^2(\Omega)} \nn \\
	& + \mu[\widehat{R_h\phi(t_{n})}] (\widehat{R_h\phi(t_{n})},v_h)_{L^2(\Omega)} - \mu[\phi_h^n] (\phi_h^n,v_h)_{L^2(\Omega)}  + \mathcal{E}(v_h) \label{equ: error equation}
\end{align}
holds for any $v_h \in S_h$.

To carry out the error estimate, we do several preparations.
\begin{lemma}\label{lem: key lemma in geometry}
	For $ 0\le n \le N$, there hold
	\begin{align}
		\|e_h^n\|_{L^2(\Omega)} &\le C \|\tilde{e}_h^n\|_{L^2(\Omega)}, \label{equ: key lemma in geometry 1}\\
		\|e_h^n\|_{L^2(\Omega)} &\le \|\tilde{e}_h^n\|_{L^2(\Omega)} + C\|\tilde{e}_h^n\|_{L^2(\Omega)}^3. \label{equ: key lemma in geometry 2}
	\end{align}
\end{lemma}
\begin{proof}
    Left in Appendix \ref{sec: appendix}.
\end{proof}
\begin{lemma}\label{lem: induction}
There exist $\tau_0$ and $h_0$ sufficiently small such that for $\tau \le \tau_0$ and $h \le h_0$, with an additional ``inverse" CFL condition $\tau \ge \kappa h^4$ where $\kappa$ is a constant sufficiently large,  such that
\begin{align}
	& \|\tilde{e}_h^n\|_{L^2(\Omega)} \le \sqrt{\tau}, \label{equ: 1 in induction} \\
	& \|\tilde{e}_h^n\|_{H^1(\Omega)} \le 1, \label{equ: 2 in induction}\\
	& \|\tilde{e}_h^n\|_{L^{\infty}(\Omega)} \le 1, \label{equ: 3 in induction}
\end{align}
for $0 \le n \le N$.
\end{lemma}

We will prove Lemma \ref{lem: induction} by mathematical induction, for $n = 0$, since $\tilde{e}_h^0 = 0$, \eqref{equ: 1 in induction}--\eqref{equ: 3 in induction} naturally hold. Assume \eqref{equ: 1 in induction}--\eqref{equ: 3 in induction} hold for $0 \le k \le n$, in the next subsection,  we will show that \eqref{equ: 1 in induction}--\eqref{equ: 3 in induction} also hold for $k = n+1$. Before that, we have the following results under assumptions of the mathematical induction.
\begin{lemma}
	For $0 \le k \le n$, the following estimates hold for $e_h^k$:
	\begin{align}\label{equ: result of induction 1}
		\|e_h^k\|_{H^1(\Omega)} \le C \|\tilde{e}_h^{k}\|_{H^1(\Omega)} + Ch^2
	\end{align}
    and 
    \begin{align}\label{equ: result of induction 2}
		\|e_h^k\|_{L^{\infty}(\Omega)} \le C.
	\end{align}
\end{lemma}
\begin{proof}
	By directly calculating and assumption \eqref{equ: 2 in induction}
	\begin{align*}
		\|e_h^k\|_{H^1(\Omega)} &\le  \|\tilde{e}_h^{k}\|_{H^1(\Omega)} + \|\phi_h^k - \tilde{\phi}_h^k\|_{H^1(\Omega)} = \|\tilde{e}_h^{k}\|_{H^1(\Omega)} + \left|\frac{1}{\|\tilde{\phi}_h^k\|_{L^2(\Omega)} } -1\right|\|\tilde{\phi}_h^k\|_{H^1(\Omega)} \\
		& \le \|\tilde{e}_h^{k}\|_{H^1(\Omega)} + C \|\phi(t_k) - \tilde{\phi}_h^k\|_{L^2(\Omega)} \\
		& \le \|\tilde{e}_h^{k}\|_{H^1(\Omega)} + C \|\phi(t_k) - \widehat{R_h\phi(t_k)}\|_{L^2(\Omega)} + C \|\widehat{R_h\phi(t_k)} - \tilde{\phi}_h^k\|_{L^2(\Omega)}  \\
		& \le C \|\tilde{e}_h^{k}\|_{H^1} + Ch^2.
	\end{align*}
This completes the proof of \eqref{equ: result of induction 1}. The proof of \eqref{equ: result of induction 2} follows from a similar argument with assumption \eqref{equ: 3 in induction}.
\end{proof}

\subsubsection{Proof of Lemma \ref{lem: induction}}

Taking $v_h = \tilde{e}_h^{n+1}$ in the error equation \eqref{equ: error equation} yields 
\begin{align*}
	& \frac{\|\tilde{e}_h^{n+1}\|_{L^2(\Omega)}^2 - \|e_h^n\|_{L^2(\Omega)}^2}{2\tau} + \|\nabla \tilde{e}_h^{n+1}\|_{L^2(\Omega)}^2 \\
	& \le \|V\|_{L^{2}(\Omega)}\|e_h^n\|_{L^3(\Omega)}\|\tilde{e}_h^{n+1}\|_{L^6(\Omega)} \\
	&\quad + C \|(\widehat{R_h\phi(t_n)})^2 + \widehat{R_h\phi(t_n)}\phi_h^n + (\phi_h^n)^2\|_{L^3(\Omega)}\|e_h^n\|_{L^2(\Omega)}\|\tilde{e}_h^{n+1}\|_{L^6(\Omega)} \\
	& \quad + |\mu[\widehat{R_h\phi(t_n)}]| |(e_h^n,\tilde{e}_h^{n+1})_{L^2(\Omega)}| + |\mu[\phi_h^n] - \mu[\widehat{R_h\phi(t_n)}]||(\phi_h^n,\tilde{e}_h^{n+1})_{L^2(\Omega)}| + |\mathcal{E}(\tilde{e}_h^{n+1})| \\
	&=: \mathcal{E}_1 + \mathcal{E}_2 + \mathcal{E}_3 + \mathcal{E}_4 + |\mathcal{E}(\tilde{e}_h^{n+1})|.
\end{align*}

Applying Young's inequality, we may estimate as follows, 
\begin{align*}
	\mathcal{E}_1 & \le C \|e_h^n\|_{L^2(\Omega)}^{1/2}\|e_h^n\|_{L^6(\Omega)}^{1/2}\|\tilde{e}_h^{n+1}\|_{H^1(\Omega)}  \le C \|\tilde{e}_h^{n}\|_{L^2(\Omega)}^{1/2}\|e_h^n\|_{H^1(\Omega)}^{1/2}\|\tilde{e}_h^{n+1}\|_{H^1(\Omega)} \\
	& \le C \|\tilde{e}_h^{n}\|_{L^2(\Omega)}^{1/2} (\|\tilde{e}_h^{n}\|_{H^1(\Omega)}^{1/2} + h)\|\tilde{e}_h^{n+1}\|_{H^1(\Omega)} \\
	& \le C_{\eps} \|\tilde{e}_h^{n}\|_{L^2(\Omega)}^2 + \eps \|\tilde{e}_h^{n}\|_{H^1(\Omega)}^2 + \eps \|\tilde{e}_h^{n+1}\|_{H^1(\Omega)}^2 + C_{\eps} h^4,
\end{align*}
here $\eps$ is a small constant that to be determined later. It is easy to see that
\begin{align*}
	\mathcal{E}_2  \le C_{\eps} \|\tilde{e}_h^{n}\|_{L^2(\Omega)}^2 + \eps \|\tilde{e}_h^{n+1}\|_{H^1(\Omega)}^2,
\end{align*}
and
\begin{align*}
	\mathcal{E}_3 & \le C \|\tilde{e}_h^n\|_{L^2(\Omega)}^2 + C \|\tilde{e}_h^{n+1}\|_{L^2(\Omega)}^2.
\end{align*}
As for $\mathcal{E}_4$, we have
\begin{align*}
	\mathcal{E}_4 & \le C \|e_h^n\|_{H^1(\Omega)}\|\tilde{e}_h^{n+1}\|_{L^2(\Omega)} \le C(\|\tilde{e}_h^n\|_{H^1(\Omega)} + h^2)\|\tilde{e}_h^{n+1}\|_{L^2(\Omega)} \\
	& \le \eps \|\tilde{e}_h^n\|_{H^1(\Omega)}^2 + C_{\eps} h^4 + C_{\eps} \|\tilde{e}_h^{n+1}\|_{L^2(\Omega)}^2.
\end{align*}

Combining estimates of $\mathcal{E}_1$--$\mathcal{E}_4$, using Lemma \ref{lem: truncation error}, equations \eqref{equ: key lemma in geometry 2} and \eqref{equ: 1 in induction},  we obtain 
\begin{align}
	& \frac{\|\tilde{e}_h^{n+1}\|_{L^2(\Omega)}^2 - \|\tilde{e}_h^n\|_{L^2(\Omega)}^2}{2\tau} + \|\nabla \tilde{e}_h^{n+1}\|_{L^2(\Omega)}^2 \nn \\
	& \le C_{\eps} (\tau^2 + h^4) + C_{\eps} (\|\tilde{e}_h^n\|_{L^2(\Omega)}^2 + \|\tilde{e}_h^{n+1}\|_{L^2(\Omega)}^2) + \eps(\|\tilde{e}_h^n\|_{H^1(\Omega)}^2 + \|\tilde{e}_h^{n+1}\|_{H^1(\Omega)}^2). \label{equ: intermediate equation}
\end{align}
Summing equation \eqref{equ: intermediate equation} from $0$ to $n+1$, we obtain
\begin{align}
    \|\tilde{e}_h^{n+1}\|_{L^2(\Omega)}^2 + \tau \sum_{k=1}^{n+1} \|\nabla \tilde{e}_h^{k}\|_{L^2(\Omega)}^2 \le C_{\eps} (\tau^2 + h^4) + C_{\eps} \tau \sum_{k=1}^{n+1} \|\tilde{e}_h^{k}\|_{L^2(\Omega)}^2 + C \eps \tau \sum_{k=1}^{n+1} \|\nabla \tilde{e}_h^{k}\|_{L^2(\Omega)}^2. \label{equ: Gronwall}
\end{align}
Choosing $\eps$ sufficiently small such that $C \eps \le \frac12$ and applying the discrete Gronwall inequality to \eqref{equ: Gronwall}, we arrive at
\begin{equation*}
	\max_{0 \le k \le n+1} \|\tilde{e}_h^k\|_{L^2(\Omega)}^2 + \tau \sum_{k = 0}^{n+1} \|\nabla \tilde{e}_h^k\|_{L^2(\Omega)}^2 \le C (\tau + h^2)^2.
\end{equation*}
Then under an ``inverse" CFL condition $\tau \ge \kappa h^4$ with $\tau$ and $h$ sufficiently small:
\begin{align*}
	& \|\tilde{e}_h^{n+1}\|_{L^2(\Omega)} \le C(\tau + h^2) \le \sqrt{\tau}, \\
	& \|\nabla \tilde{e}_h^{n+1}\|_{L^2(\Omega)} \le C (\sqrt{\tau} + \frac{h^2}{\sqrt{\tau}}) \le 1.
\end{align*}
This completes the proof of \eqref{equ: 1 in induction} and \eqref{equ: 2 in induction} for $k = n+1$. 
If we are able to prove
\begin{equation*}\tag{$*$}
	\|\Delta_h \tilde{e}_h^{n+1}\|_{L^2(\Omega)} \le C,
\end{equation*}
we then obtain from \eqref{equ: discrete interpolation inequality} that
\begin{align*}
	\|\tilde{e}_h^{n+1}\|_{L^{\infty}(\Omega)} & \le \|\tilde{e}_h^{n+1}\|_{L^2(\Omega)}^{1/4}\|\Delta_h \tilde{e}_h^{n+1}\|_{L^2(\Omega)}^{3/4} \le C (\tau + h^2)^{1/4},
\end{align*}
which proves \eqref{equ: 3 in induction} for $k = n+1$, and the proof of mathematical induction is complete.

To prove ($*$), we divide into two cases:
\begin{equation*}
	\tau \le h^{2} \quad \text{and} \quad \tau \ge h^{2}.
\end{equation*}
If $\tau \le h^{2}$, by the inverse inequality, we can get 
\begin{equation*}
	\|\Delta \tilde{e}_h^{n+1}\|_{L^2(\Omega)} \le C h^{-2} \|\tilde{e}_h^{n+1}\|_{L^2(\Omega)} \le C h^{-2} (\tau + h^2) \le C.
\end{equation*}
If $\tau \ge h^{2}$, we rewrite the error equation \eqref{equ: error equation} as 
\begin{align*}
	 (\Delta_h \tilde{e}_h^{n+1}, v_h)_{L^2(\Omega)} & = \left( \frac{\tilde{e}_h^{n+1} - e_h^n}{\tau}, v_h \right)_{L^2(\Omega)}   + (V e_h^n, v_h)_{L^2(\Omega)} \\
	 & \quad + (\beta (\widehat{R_h\phi(t_n)})^3 - \beta (\phi_h^n)^3, v_h)_{L^2(\Omega)} \\
	& \quad  - \mu[\widehat{R_h\phi(t_n)}] (\widehat{R_h\phi(t_n)},v_h)_{L^2(\Omega)} + \mu[\phi_h^n] (\phi_h^n,v_h)_{L^2(\Omega)}   - \mathcal{E}(v_h) \\
	& =: \mathcal{G}_1 + \mathcal{G}_2 + \mathcal{G}_3 + \mathcal{G}_4 + \mathcal{G}_5 + \mathcal{G}_6.
\end{align*}
We estimate $\mathcal{G}_1$ as follows 
\begin{align*}
	|\mathcal{G}_1|& \le \left\| \frac{\tilde{e}_h^{n+1} - e_h^n}{\tau} \right\|_{L^2(\Omega)} \|v_h\|_{L^2(\Omega)} \le \tau^{-1} (\|\tilde{e}_h^{n+1}\|_{L^2(\Omega)} + \|e_h^{n}\|_{L^2(\Omega)})\|v_h\|_{L^2(\Omega)} \\
	& \le C \tau^{-1} (\tau + h^2) \|v_h\|_{L^2(\Omega)} \le C \|v_h\|_{L^2(\Omega)}.
\end{align*}
Obviously, 
\begin{align*}
	|\mathcal{G}_2| \le C \|V\|_{L^2(\Omega)}\|e_h^n\|_{L^{\infty}(\Omega)}\|v_h\|_{L^2(\Omega)} \le C \|v_h\|_{L^2(\Omega)} 
\end{align*}
and
\begin{equation*}
	|\mathcal{G}_j| \le C \|v_h\|_{L^2(\Omega)}, \qquad j = 3,\cdots,5.
\end{equation*}
With slightly modified the estimates of the truncation error, we have 
\begin{equation*}
	|\mathcal{G}_6| \le C \|v_h\|_{L^2(\Omega)}.
\end{equation*}
Estimates of $\mathcal{G}_1$--$\mathcal{G}_6$ imply 
\begin{align*}
	\|\Delta_h \tilde{e}_h^{n+1}\|_{L^2(\Omega)} \le C.
\end{align*}
The proof is completed by combing estimates of two cases.

\subsubsection{Error estimate}

From the proof of Lemma \ref{lem: induction}, we have the following convergence results.
\begin{theorem}\label{thm: L2 convergence}
	Under the same assumptions in Lemma \ref{lem: induction}, there holds
	\begin{equation*}
		\max_{0 \le n \le N} \|\phi(t_n) - \phi_h^n\|_{L^2(\Omega)} \le C_{\kappa}(\tau + h^2).
	\end{equation*}
\end{theorem}
\begin{proof}
	We obtain from Lemma \ref{lem: key lemma in geometry} that
	\begin{equation*}
		\|e_h^n\|_{L^2(\Omega)} \le C \|\tilde{e}_h^n\|_{L^2(\Omega)} \le C(\tau + h^2),
	\end{equation*}
    which leads to
	\begin{align*}
		\|\phi(t_n) - \phi_h^n\|_{L^2(\Omega)} & \le \|\phi(t_n) - \widehat{R_h\phi(t_n)}\|_{L^2(\Omega)} +  \|e_h^n\|_{L^2(\Omega)} \\
		& \le C h^2 + C(\tau + h^2) \le C(\tau + h^2),
	\end{align*}
	and completes the proof.
\end{proof}
\begin{theorem}[discrete $H^2$ stability]
	Under the same assumptions in Lemma \ref{lem: induction}, there holds
	\begin{equation*}
		\max_{0 \le n \le N}\|\Delta_h \phi_h^n\|_{L^2(\Omega)} \le C_{\kappa}.
	\end{equation*}
\end{theorem}
\begin{proof}
	Note that 
	\begin{equation*}
		\phi_h^n = \widehat{\tilde{\phi}_h^n},
	\end{equation*}
	we obtain from ($*$) that
	\begin{align*}
		\|\Delta_h \phi_h^n\|_{L^2(\Omega)} &\le C \|\Delta_h \tilde{\phi}_h^n\|_{L^2(\Omega)} \\
		& \le C \|\Delta_h \widehat{R_h\phi(t_n)}\|_{L^2(\Omega)} + C \|\Delta_h \tilde{e}_h^n\|_{L^2(\Omega)} \le C,
	\end{align*}
	which completes the proof.
\end{proof}
\begin{corollary}\label{thm: H1 convergence}
	Under the same assumptions in Lemma \ref{lem: induction}, there holds
	\begin{equation*}
		\max_{0 \le n \le N} \|\phi(t_n) - \phi_h^n\|_{H^1(\Omega)} \le C_{\kappa} (\sqrt{\tau} + h).
	\end{equation*}
\end{corollary}
\begin{proof}
	With inequality \eqref{equ: discrete interpolation}, we obtain
	\begin{equation*}
		\|\nabla e_h^n\|_{L^2(\Omega)} \le \|e_h^n\|_{L^2(\Omega)}^{\frac12}\|\Delta_h e_h^n\|_{L^2(\Omega)}^{\frac12} \le C (\sqrt{\tau} + h),
	\end{equation*}
	which leads to 
	\begin{align*}
		\|\phi(t_n) - \phi_h^n\|_{H^1(\Omega)} & \le  \| \phi(t_n) - \widehat{R_h\phi(t_n)}\|_{H^1(\Omega)} + C \|\nabla e_h^n\|_{L^2(\Omega)} \\
		& \le C (\sqrt{\tau} + h).
	\end{align*}
	This completes the proof.
\end{proof}

\section{Numerical experiments}\label{sec: numerical experiments}

In this section, we will provide two numerical experiments to verify our convergence results of numerical scheme \eqref{equ: numerical scheme 1}--\eqref{equ: numerical scheme 3} corresponding to the gradient flow model, i.e., Theorem \ref{thm: L2 convergence} and Corollary \ref{thm: H1 convergence}. All computations are performed on $\Omega  = (0,1)^3$ up to time $T = 1$. The numerical error is quantified in the following norms: 
\begin{align*}
    e_{L^2} = \|\phi_h^N - \phi_{\text{ref}}\|_{L^2(\Omega)} \qquad e_{H^1} = \|\phi_h^N - \phi_{\text{ref}}\|_{H^1(\Omega)},
\end{align*}
where $\phi_{\text{ref}}$ denotes a reference solution computed by using the same numerical scheme as for $\phi_h^N$ with a finer time step or spatial mesh size. The corresponding code is designed based on the toolbox PHG (Parallel Hierarchical Grid) \cite{PHG} and the computations are carried out on LSSC--IV cluster in the State Key Laboratory of Mathematical Sciences of the Chinese Academy of
Sciences.

\subsection{Example I. Harmonic potential}

We consider the system \eqref{equ: PDEofL2GF} with $\beta =10$, the initial function $\phi_0$ and potential function $V$ are as follows:
\begin{equation}
   \begin{aligned}
       \phi_0 =cx(1-x)y(1-y)z(1-z), \qquad
        V=\frac{1}{2}(x^2+y^2+z^2),
   \end{aligned}
\end{equation}
where $c>0$ is a constant such that $\Vert\phi_0\Vert_{L^2(\Omega)}=1$.

Tables \ref{tab1} and \ref{tab2} show the temporal discretization errors $e_{L^2}$ and $e_{H^1}$ and the temporal convergence orders, respectively. In each case, the reference solution $\phi_{\text{ref}}$ is computed using the scheme \eqref{equ: numerical scheme 1}--\eqref{equ: numerical scheme 3} with the same spatial mesh as that used in computing $\phi_h^N$ and a smaller time step $\tau = \frac{1}{2000}$. 

Tables \ref{tab3} and \ref{tab4} show the spatial discretization errors and convergence orders of $e_{L^2}$ and $e_{H^1}$, respectively. In each case, the reference solution $\phi_{\text{ref}}$ is computed using the scheme \eqref{equ: numerical scheme 1}--\eqref{equ: numerical scheme 3} with the same temporal mesh as that used in computing $\phi_h^N$ and a smaller spatial mesh $h = 0.027063$. 

From Tables \ref{tab1}--\ref{tab4}, we can see that the scheme exhibits first-order temporal convergence under the $L^2$ and $H^1$ norms, first-order spatial convergence under the $H^1$ norm, and second-order spatial convergence under the $L^2$ norm.

\begin{table}[H]\label{tab1}
    \setlength{\tabcolsep}{30pt}
    \renewcommand{\arraystretch}{0.95}
    \centering
    \caption{\parbox{0.9\textwidth}{
    \begin{center}
        $L^2$ temporal discretization errors of scheme \eqref{equ: numerical scheme 1}--\eqref{equ: numerical scheme 3} with reference time step $\tau=\frac{1}{2000}$.
        \end{center}}}
    %\label{tab:discretization_errors}
    \begin{tabular}{cccc}
        \hline
        \rule{0pt}{2ex}
        $h$ & $\tau$ & $e_{L^2}$ & Order \\ 
        \hline
        \rule{0pt}{3ex}
         & 1/90 & 2.0951E-02 & - \\ 
        0.216506  & 1/180 & 1.0940E-02 & 0.9374 \\ 
          & 1/360 & 5.1690E-03 & 1.0817 \\ 
        & ~ & ~ & ~\\
         & 1/90 & 2.0952E-02 & - \\ 
         0.108253 & 1/180 & 1.0972E-02 & 0.9333\\ 
          & 1/360 & 5.1920E-03 & 1.0795 \\ 
        \hline
    \end{tabular}
\end{table}

\begin{table}[H]\label{tab2}
    \setlength{\tabcolsep}{30pt}
    \renewcommand{\arraystretch}{0.95}
    \centering
    \caption{\parbox{0.9\textwidth}{
    \begin{center}
    $H^1$ temporal discretization errors of scheme \eqref{equ: numerical scheme 1}--\eqref{equ: numerical scheme 3} with reference time step $\tau=\frac{1}{2000}$.
    \end{center}}}
    %\label{tab:discretization_errors}
    \begin{tabular}{cccc}
        \hline
        \rule{0pt}{2ex}
        $h$ & $\tau$ & $e_{H^1}$ & Order \\ 
        \hline
        \rule{0pt}{3ex}
         & 1/90 & 2.4147E-01 & - \\ 
        0.216506  & 1/180 & 1.2583E-01 & 0.9404 \\ 
          & 1/360 & 5.9384E-02 & 1.0833 \\ 
        & ~ & ~ & ~\\
         & 1/90 & 2.3044E-01 & - \\ 
         0.108253 & 1/180 & 1.2037E-01 & 0.9369\\ 
          & 1/360 & 5.6883E-02 & 1.0814 \\ 
        \hline
    \end{tabular}
\end{table}

\begin{table}[H]\label{tab3}
    \setlength{\tabcolsep}{30pt}
    \renewcommand{\arraystretch}{0.95}
    \centering
    \caption{\parbox{0.9\textwidth}{
    \begin{center}
    $L^2$ spatial discretization errors of scheme \eqref{equ: numerical scheme 1}--\eqref{equ: numerical scheme 3} with reference spatial mesh size $h=0.027063$.
    \end{center}}}
    %\label{tab:discretization_errors}
    \begin{tabular}{cccc}
        \hline
        \rule{0pt}{2ex}
        $\tau$ & $h$ & $e_{L^2}$ & Order \\ 
        \hline
        \rule{0pt}{3ex}
         & 0.433013 & 1.4300E-01 & - \\ 
        $\frac{1}{1000}$  & 0.216506 & 3.7730E-02  & 1.9222 \\ 
          & 0.108253 & 9.1250E-03 & 2.0478 \\ 
        & ~ & ~ & ~\\
         & 0.433013 & 1.4298E-01 & - \\ 
        $\frac{1}{2000}$  & 0.216506 & 3.7687E-02  & 1.9237 \\ 
          & 0.108253 & 9.1241E-03 & 2.0463 \\ 
        \hline
    \end{tabular}
\end{table}

\begin{table}[H]\label{tab4}
    \setlength{\tabcolsep}{30pt}
    \renewcommand{\arraystretch}{0.95}
    \centering
    \caption{\parbox{0.9\textwidth}{
    \begin{center}
    $H^1$ spatial discretization errors of scheme \eqref{equ: numerical scheme 1}--\eqref{equ: numerical scheme 3} with reference spatial mesh size $h=0.027063$ and corresponding orders.
    \end{center}}}
    %\label{tab:discretization_errors}
    \begin{tabular}{cccc}
        \hline
        \rule{0pt}{2ex}
        $\tau$ & $h$ & $e_{H^1}$ & Order \\ 
        \hline
        \rule{0pt}{3ex}
         & 0.433013 & 2.9615 & - \\ 
        $\frac{1}{1000}$  & 0.216506 & 1.4438  & 1.0365 \\ 
          & 0.108253 & 0.6971 & 1.0504 \\ 
        & ~ & ~ & ~\\
         & 0.433013 & 2.9599 & - \\ 
        $\frac{1}{2000}$  & 0.216506 & 1.4419  & 1.0376 \\ 
          & 0.108253 & 0.6962 & 1.0504 \\ 
        \hline
    \end{tabular}
\end{table}

\subsection{Example II. Lattice potential} 

We consider the system \eqref{equ: PDEofL2GF} with $\beta=10$, the initial function $\phi_0$ and potential function $V$ are as follows:
\begin{align*}
    &\phi_0 =cx(1-x)y(1-y)z(1-z),\\
    &V =\frac{(x^2+y^2+z^2)}{2}+20+20\sin(2\pi x)\sin(2\pi y)\sin(2\pi z),
\end{align*}
where $c>0$ is a constant such that $\Vert\phi_0\Vert_{L^2(\Omega)}=1$.

We present the temporal discretization errors and spatial discretization errors of the scheme \eqref{equ: numerical scheme 1}--\eqref{equ: numerical scheme 3} in Tables \ref{tab5}--\ref{tab8}. 

The temporal convergence orders of $e_{L^2}$ and $e_{H^1}$ are shown in Tables \ref{tab5} and \ref{tab6}. The reference solution $\phi_{\text{ref}}$ is computed by the scheme \eqref{equ: numerical scheme 1}--\eqref{equ: numerical scheme 3} with the same spatial mesh as in the computation of $\phi_h^N$ and a smaller time step $\tau = \frac{1}{2000}$. Tables \ref{tab5} and \ref{tab6} show spatial convergence orders of $e_{L^2}$ and $e_{H^1}$ while the reference solution $\phi_{\text{ref}}$ is derived with the same temporal mesh in computing $\phi_h^N$ and a smaller spatial mesh $h = 0.017469$.

Similar to Example I, Tables \ref{tab5}--\ref{tab8} show that the convergence order of $e_{L^2}$ is $O(\tau + h^2)$, which validates the conclusion in Theorem \ref{thm: L2 convergence}. The convergence order of $e_{H^1}$ is $O(\tau + h)$, showing that the result in Corollary \ref{thm: H1 convergence} is sub-optimal.

\begin{table}[H]\label{tab5}
    \setlength{\tabcolsep}{30pt}
    \renewcommand{\arraystretch}{0.95}
    \centering
    \caption{\parbox{0.9\textwidth}{
    \begin{center}
    $L^2$ temporal discretization errors of scheme \eqref{equ: numerical scheme 1}--\eqref{equ: numerical scheme 3} with reference time step $\tau=\frac{1}{2000}$.
    \end{center}}}
    %\label{tab:discretization_errors}
    \begin{tabular}{cccc}
        \hline
        \rule{0pt}{2ex}
        $h$ & $\tau$ & $e_{L^2}$ & Order \\ 
        \hline
        \rule{0pt}{3ex}
         & 1/70 & 5.9623E-02 & - \\ 
        0.216506  & 1/140 & 3.0171E-02 & 0.9827 \\ 
          & 1/280 & 1.4200E-02 & 1.0873 \\ 
        & ~ & ~ & ~\\
         & 1/70 & 6.2760E-02 & - \\ 
         0.108253 & 1/140 & 3.2162E-02 & 0.9645\\ 
          & 1/280 & 1.5221E-03 & 1.0793 \\ 
        \hline
    \end{tabular}
\end{table}

\begin{table}[H]\label{tab6}
    \setlength{\tabcolsep}{30pt}
    \renewcommand{\arraystretch}{0.95}
    \centering
    \caption{\parbox{0.9\textwidth}{
    \begin{center}
    $H^1$ temporal discretization errors of scheme \eqref{equ: numerical scheme 1}--\eqref{equ: numerical scheme 3} with reference time step $\tau=\frac{1}{2000}$.
    \end{center}}}
    %\label{tab:discretization_errors}
    \begin{tabular}{cccc}
        \hline
        \rule{0pt}{2ex}
        $h$ & $\tau$ & $e_{H^1}$ & Order \\ 
        \hline
        \rule{0pt}{3ex}
         & 1/70 & 6.9833E-01 & - \\ 
        0.216506  & 1/140 & 3.5162E-01 & 0.9899 \\ 
          & 1/280 & 1.6507E-01 & 1.0909 \\ 
        & ~ & ~ & ~\\
         & 1/70 & 6.9649E-01 & - \\ 
         0.108253 & 1/140 & 3.5514E-01 & 0.9717\\ 
          & 1/280 & 1.6765E-01 & 1.0829 \\ 
        \hline
    \end{tabular}
\end{table}

\begin{table}[H]\label{tab7}
    \setlength{\tabcolsep}{30pt}
    \renewcommand{\arraystretch}{0.95}
    \centering
    \caption{\parbox{0.9\textwidth}{
    \begin{center}
    $L^2$ spatial discretization errors of scheme \eqref{equ: numerical scheme 1}--\eqref{equ: numerical scheme 3} with reference spatial mesh size $h = 0.017469$.
    \end{center}}}
    %\label{tab:discretization_errors}
    \begin{tabular}{cccc}
        \hline
        \rule{0pt}{2ex}
        $\tau$ & $h$ & $e_{L^2}$ & Order \\ 
        \hline
        \rule{0pt}{3ex}
         & 0.279508 & 6.0835E-02  & - \\ 
        $\frac{1}{500}$  & 0.139754 & 1.4965E-02  & 2.0233 \\ 
          & 0.069877 & 3.6470E-03 & 2.0368 \\ 
        & ~ & ~ & ~\\
         & 0.279508 & 5.9719E-02 & - \\ 
        $\frac{1}{1000}$  & 0.139754 & 1.4734E-02  & 2.0190 \\ 
          & 0.069877 & 3.6120E-03 & 2.0283 \\ 
        \hline
    \end{tabular}
\end{table}

\begin{table}[H]\label{tab8}
    \setlength{\tabcolsep}{30pt}
    \renewcommand{\arraystretch}{0.95}
    \centering
    \caption{\parbox{0.9\textwidth}{
    \begin{center}
    $H^1$ spatial discretization errors of scheme \eqref{equ: numerical scheme 1}--\eqref{equ: numerical scheme 3} with reference spatial mesh size $h = 0.017469$.
    \end{center}}}
    %\label{tab:discretization_errors}
    \begin{tabular}{cccc}
        \hline
        \rule{0pt}{2ex}
        $\tau$ & $h$ & $e_{H^1}$ & Order \\ 
        \hline
        \rule{0pt}{3ex}
         & 0.279508 & 1.6535  & - \\ 
        $\frac{1}{500}$&0.139754  & 7.9726E-01  & 1.0524 \\ 
          & 0.069877 & 3.8740E-01 & 1.0412 \\ 
        & ~ & ~ & ~\\
         & 0.279508 & 1.6381 & - \\ 
        $\frac{1}{1000}$ &0.139754 & 7.9058E-01  & 1.0510 \\ 
          & 0.069877 & 3.8341E-01 & 1.0440 \\ 
        \hline
    \end{tabular}
\end{table}

\section{Conclusion}\label{sec: conclusion}

In this paper, we have carried out the mathematical analysis of the $L^2$ normalized gradient flow model for the GP eigenvalue problem and given a normalized implicit-explicit full-discrete scheme. We prove the well-posedness of the scheme and establish the optimal-order convergence of the approximations. Under a mild inverse CFL condition, we have proven the convergence rate is of order $O(\tau + h^2)$ in $L^{\infty}(0,T;L^2(\Omega))$-norm. With this convergence result, combining the $H^2$-stability, we have proved that the numerical scheme has a sub-optimal convergence of the form $O(\sqrt{\tau} + h)$ in $L^{\infty}(0,T;H^1(\Omega))$-norm. Numerical experiments validate our theories and show that the theoretical convergence order in $L^{\infty}(0,T;H^1(\Omega))$-norm is not sharp. The analysis of the optimal-order convergence in $L^{\infty}(0,T;H^1(\Omega))$-norm will be left as future work.

\appendix

\section{Properties of discrete operators}

\begin{lemma}[\cite{brenner2008mathematical}]
For the Ritz projection operator $R_h$ and $u \in H^2(\Omega) \cap H_0^1(\Omega)$, there hold:
\begin{align}
	 \|\nabla ( u - R_h u )\|_{L^2(\Omega)} &\le C h \|D^2u\|_{L^2(\Omega)}, \\
	 \|u - R_hu\|_{L^2(\Omega)} &\le Ch^2 \|D^2u\|_{L^2(\Omega)}. \label{equ: L2 convergence of Ritz projection}
\end{align}
\end{lemma}
\begin{lemma}\label{lem: L infty boundedness of R_h}
If $u \in H^2(\Omega) \cap H_0^1(\Omega)$, then 
	\begin{equation}
		\|R_h u\|_{L^{\infty}(\Omega)} \le C \|u\|_{H^2(\Omega)}.
	\end{equation}
\end{lemma}
\begin{proof}
	Denote $I_hu \in S_h$ the interpolation of $u$,  we have 
	\begin{align*}
		\|R_h u\|_{L^{\infty}(\Omega)} & \le \|I_hu\|_{L^{\infty}(\Omega)} + \|I_hu - R_h u\|_{L^{\infty}(\Omega)} \\
		& \le \|u\|_{L^{\infty}(\Omega)} + C h^{-\frac32}\|I_hu - R_h u\|_{L^{2}(\Omega)} \quad (\text{inverse inequality}) \\
		& \le C \|u\|_{H^{2}(\Omega)} + C h^{-\frac32}\|I_hu -  u\|_{L^{2}(\Omega)} + C h^{-\frac32}\|u - R_h u\|_{L^{2}(\Omega)} \\
		& \le C \|u\|_{H^2(\Omega)},
	\end{align*}
	where we have used the Sobolev embedding inequality and the approximation property of the interpolation operator \cite{brenner2008mathematical} in the last line.
\end{proof}
\begin{lemma}\label{lem: discrete Laplacian operator}
	The discrete Laplacian operator satisfies 
	\begin{equation}\label{equ: discrete interpolation}
		\|\nabla u_h \|_{L^2(\Omega)} \le \|u_h\|_{L^2(\Omega)}^{\frac12} \|\Delta_h u_h\|_{L^2(\Omega)}^{\frac12} \qquad \forall \, u_h \in S_h,
	\end{equation}
    and
    \begin{equation}\label{equ: discrete interpolation inequality}
		\|u_h \|_{L^{\infty}(\Omega)} \le C\|u_h\|_{L^2(\Omega)}^{\frac14} \|\Delta_h u_h\|_{L^2(\Omega)}^{\frac34} \qquad \forall \, u_h \in S_h.
	\end{equation}
\end{lemma}
\begin{proof}
By the definition of $\Delta_h$ and H\"older's inequality, we obtain
\begin{equation*}
 	\|\nabla u_h\|_{L^2(\Omega)}^2 = - (\Delta_h u_h, u_h)_{L^2(\Omega)} \le \|u_h\|_{L^2(\Omega)} \|\Delta_h u_h\|_{L^2(\Omega)},
 \end{equation*}
which completes the proof of \eqref{equ: discrete interpolation}.

To prove \eqref{equ: discrete interpolation inequality}, let $u \in H^2(\Omega) \cap H_0^1(\Omega)$ such that 
\begin{equation*}
	\Delta u = \Delta_h u_h.
\end{equation*}
Then we have $\|u\|_{H^2(\Omega)} \le C \|\Delta_h u_h \|_{L^2(\Omega)}$. Next we show that $\|u\|_{L^2(\Omega)} \le C \|u_h\|_{L^2(\Omega)}$.
	
Note that 
\begin{align}
	\|u\|_{L^2(\Omega)} & \le \|u_h\|_{L^2(\Omega)} + \|u - u_h\|_{L^2(\Omega)} \nn \\
	& \le  \|u_h\|_{L^2(\Omega)} + h^2 \|u\|_{H^2(\Omega)} \qquad (\text{here \eqref{equ: L2 convergence of Ritz projection} is used}) \nn \\
	& \le \|u_h\|_{L^2(\Omega)} + C h^2 \|\Delta_h u_h \|_{L^2(\Omega)} \nn \\
	& \le C \|u_h\|_{L^2(\Omega)}, \label{equ: 4.6}
\end{align}
where the inverse inequality is applied to obtain the last line. Similarly, we can prove that
\begin{align}\label{equ: 4.7}
\|I_hu\|_{L^2(\Omega)} \le C \|u_h\|_{L^2(\Omega)}.
\end{align}
	
Then by the triangle inequality, we get 
\begin{align*}
	\|u_h \|_{L^{\infty}(\Omega)} & \le \|u \|_{L^{\infty}(\Omega)} + \|u - u_h \|_{L^{\infty}(\Omega)} \\
	& \le \|u_h \|_{L^{\infty}(\Omega)} + \|u - I_hu \|_{L^{\infty}(\Omega)} + \|I_hu - u_h \|_{L^{\infty}(\Omega)} \\
	& \le C \|u\|_{L^{\infty}(\Omega)} + \|I_hu - u_h \|_{L^{\infty}(\Omega)} \\
	& \le C \|u\|_{L^2(\Omega)}^{\frac14}\|u\|_{H^2(\Omega)}^{\frac34} + h^{-\frac32} \|I_hu - u_h \|_{L^{2}(\Omega)},  
    \end{align*}
which together with Sobolev interpolation inequality, inverse inequality, equations \eqref{equ: 4.6} and \eqref{equ: 4.7} yields
    \begin{align*}
	\|u_h \|_{L^{\infty}(\Omega)}& \le C \|u_h\|_{L^2(\Omega)}^{\frac14}\|\Delta_h u\|_{L^2(\Omega)}^{\frac34} + h^{-\frac32} \|I_hu - u_h \|_{L^{2}(\Omega)}^{\frac14} \|I_hu - u_h \|_{L^{2}(\Omega)}^{\frac34} \\
	& \le C \|u_h\|_{L^2(\Omega)}^{\frac14}\|\Delta_h u\|_{L^2(\Omega)}^{\frac34} \\
	& \quad + h^{-\frac32} \|I_hu - u_h \|_{L^{2}(\Omega)}^{\frac14} \big( \|I_hu - u \|_{L^{2}(\Omega)} + \|u - u_h \|_{L^{2}(\Omega)}\big)^{\frac34}  \\
	& \le C \|u_h\|_{L^2(\Omega)}^{\frac14}\|\Delta_h u\|_{L^2(\Omega)}^{\frac34} + h^{-\frac32} \| u_h \|_{L^{2}(\Omega)}^{\frac14} \big( h^2 \|\Delta_h u_h \|_{L^2(\Omega)}\big)^{\frac34} \\
	& \le C \|u_h\|_{L^2(\Omega)}^{\frac14}\|\Delta_h u\|_{L^2(\Omega)}^{\frac34}.
\end{align*}
This completes the proof of \eqref{equ: discrete interpolation inequality}.
\end{proof}

\section{Detailed proofs}\label{sec: appendix}

\subsection{Proof of Lemma \ref{lem: technical lemma 1}}

\begin{proof}
	A direct calculation shows that 
	\begin{align*}
		\mu[u] - \mu[v] & = \int_{\Omega} \Big(  (|\nabla u|^2 - |\nabla v|^2) + V(u^2 - v^2) + \beta (u^4 - v^4) \Big) \d x \\
		& = \int_{\Omega} \Big(  (\nabla u + \nabla v)(\nabla u - \nabla v) + V (u+v)(u-v) \\
		& \quad\quad + \beta (u^2 + v^2)(u+v)(u-v) \Big) \d x.
	\end{align*}
	We then obtain from H\"older's inequality that 
	\begin{align*}
		|\mu[u] - \mu[v]| & \le M \|u-v\|_{H^1(\Omega)} + \|V\|_{L^{2}(\Omega)}\|u+v\|_{L^4(\Omega)}\|u-v\|_{L^4(\Omega)} \\
		& \quad + \beta \|u^2 + v^2\|_{L^3(\Omega)}\|u+v\|_{L^6(\Omega)}\|u-v\|_{L^2(\Omega)} \\
		& \le C_M \|u-v\|_{H^1(\Omega)},
	\end{align*} 
	where the Sobolev embedding $H^1(\Omega) \hookrightarrow L^6(\Omega)$ has been used.
\end{proof}

\subsection{Proof of Lemma \ref{lem: technical lemma 2}}

\begin{proof}
	Denote 
\begin{equation*}
	A:= \int_{\Omega} \Big( |\nabla R_hu|^2 + V |R_h u|^2 + \beta |R_hu|^4 \Big)\ \text{d}x.
\end{equation*}
Then we obtain
\begin{align*}
		|\mu[u] - A | & \le \left| \int_{\Omega}  (|\nabla u|^2 - |\nabla R_hu|^2)\  \text{d} x \right| + \|V\|_{L^{2}(\Omega)}\|u+R_hu\|_{L^{\infty}(\Omega)}\|u-R_hu\|_{L^2(\Omega)} \\
		& \quad + \beta \|u^2 + (R_hu)^2\|_{L^3(\Omega)}\|u+R_hu\|_{L^6(\Omega)}\|u-R_hu\|_{L^2(\Omega)} \\
		& =: I_1 + I_2 + I_3.
	\end{align*}
Properties of the Ritz projection show that
	\begin{align*}
		I_2 & \le C_1 h^2  , \qquad I_3 \le  C_2 h^2,
	\end{align*}
    here constants $C_1$ and $C_2$ depend on $\|u\|_{H^2(\Omega)}$.
	For $I_1$, we get
\begin{align*}
		I_1 & = \left| \int_{\Omega} (\nabla u + \nabla R_hu)\cdot(\nabla u - \nabla R_hu) \ \text{d} x \right| = \left| \int_{\Omega}   \nabla u (\nabla u - \nabla R_hu) \  \text{d} x \right| \\
		& = \left|\int_{\Omega} \Delta u (u -  R_hu) \ \text{d} x\right| \le C_3h^2,
	\end{align*}
    with $C_3$ depends on $\|u\|_{H^2(\Omega)}$.
    
	The left is the estimate of $\left|A - \mu[\widehat{R_hu}]\right|$, 
	\begin{align*}
		\left|A - \mu[\widehat{R_hu}]\right| & \le \left| 1 - \frac{1}{\|R_hu\|_{L^2(\Omega)}^2} \right|\left| \int_{\Omega} \Big(  |\nabla R_hu|^2 + V |R_hu|^2 \Big) \ \text{d} x\right| \\
		& \quad  + \left| 1 - \frac{1}{\|R_hu\|_{L^2(\Omega)}^4} \right| \left| \int_{\Omega} \beta |R_hu|^4 \ \text{d} x\right| \\
		& \le C_4h^2,
	\end{align*}
	here $C_4$ depends on $\|u\|_{H^2(\Omega)}$. Combining the above estimates, we complete the proof.
\end{proof}

\subsection{Proof of Lemma \ref{lem: key lemma in geometry}}

\begin{proof}
	Denote
    \begin{align*}
    \alpha=\arccos\frac{(R_h\phi(t_n),\phi_h^n)_{L^2(\Omega)}}{\|R_h\phi(t_n)\|_{L^2(\Omega)}\|\phi_h^n\|_{L^2(\Omega)}}
    \end{align*}
    be the angle between vectors $R_h\phi(t_n)$ and $\phi_h^n$, then for $\alpha > \pi/6$, we have $\|\tilde{e}_h^n\|_{L^2(\Omega)} \ge 1/2$, with the fact $\|e_h^n\|_{L^2(\Omega)} \le 2$, we can complete the proof of \eqref{equ: key lemma in geometry 1}--\eqref{equ: key lemma in geometry 2} for the case $\alpha > \pi/6$. 
	
	For the case $0 \le \alpha \le \pi/6$, we have $\|\tilde{e}_h^n\|_{L^2(\Omega)} \ge \sin\alpha$, and 
	\begin{align*}
		\|e_h^n\|_{L^2(\Omega)} = 2\sin(\alpha/2) \le \alpha \le C \sin\alpha \le C \|\tilde{e}_h^n\|_{L^2(\Omega)}.
	\end{align*}
	Meanwhile, 
	\begin{align*}
		\|e_h^n\|_{L^2(\Omega)} - \|\tilde{e}_h^n\|_{L^2(\Omega)} &  \le \|e_h^n\|_{L^2(\Omega)} - \|e_h^n\|_{L^2(\Omega)}\cos(\alpha/2) \\
		& = \|e_h^n\|_{L^2(\Omega)} 2\sin^2(\alpha/4) \le C \|\tilde{e}_h^n\|_{L^2(\Omega)}^3,
	\end{align*}
	which completes the proof.
\end{proof}

% References
\bibliographystyle{siamplain}
\bibliography{references}

\end{document}